\documentclass[reqno]{amsart}
\usepackage{xcolor}
\usepackage[utf8]{inputenc}
\usepackage{hyperref}
\usepackage{enumitem}

\long\def\MSC#1\EndMSC{\def\arg{#1}\ifx\arg\empty\relax\else
     {\narrower\noindent%
{2010 Mathematics Subject Classification}: #1\\} \fi}
\long\def\PACS#1\EndPACS{\def\arg{#1}\ifx\arg\empty\relax\else
     {\narrower\noindent%
{PACS numbers}: #1}\fi}
\long\def\KEY#1\EndKEY{\def\arg{#1}\ifx\arg\empty\relax\else
	{\narrower\noindent%
Keywords: #1\\}\fi}

\usepackage{a4wide,amsmath}
\usepackage{amssymb}
\usepackage{mathtools}
\usepackage{mathrsfs}
\usepackage{amsthm}
\numberwithin{equation}{section}

\theoremstyle{plain}
\newtheorem{theorem}{Theorem}[section]
\newtheorem{lemma}[theorem]{Lemma}
\newtheorem{proposition}[theorem]{Proposition}
\newtheorem{corollary}[theorem]{Corollary}
\theoremstyle{definition}

\theoremstyle{remark}
\newtheorem{remark}[theorem]{Remark}
\newcommand{\Cinf}{C^\infty}

\newcommand{\No}{\mathbb{N}_0}
\newcommand{\N}{\mathbb{N}}	
\newcommand{\Z}{\mathbb{Z}}
\newcommand{\R}{\mathbb{R}}
\newcommand{\C}{\mathbb{C}}
\newcommand{\SR}{\mathscr{S}(\R^d)}
\newcommand{\SpR}{\mathscr{S}'(\R^d)}
\newcommand{\sH}{\mathscr{H}}

\newcommand{\cO}{\mathcal{O}}
\newcommand{\sC}{\mathscr{C}}
\newcommand{\Mphi}{M_\phi(\R^{3d})}
\newcommand{\cA}{\mathcal{A}}
\newcommand{\abs}[1]{\vert #1\vert}
\newcommand{\babs}[1]{\big\vert #1\big\vert}

\newcommand{\babss}[1]{\bigg\vert #1\bigg\vert}
\newcommand{\pair}[2]{\langle #1,#2\rangle}
\newcommand{\ip}[2]{\langle #1,#2\rangle}
\newcommand{\norm}[1]{\Vert #1\Vert}
\newcommand{\Bnorm}[1]{\Big\Vert #1\Big\Vert}

\newcommand{\Op}{\mathrm{Op}}
\newcommand{\jn}[1]{\langle#1\rangle}
\newcommand{\fl}{\mathfrak{f}}

\DeclareMathOperator{\id}{id}
\newcommand{\sF}{\mathscr{F}}
\DeclareMathOperator{\Supp}{supp}
\newcommand{\dist}{\mathrm{dist}}
\newcommand{\dH}{d_\mathrm{H}}  
\newcommand{\I}{\mathrm{i}}     
\newcommand{\e}{\mathrm{e}}     
\newcommand{\di}{\mathrm{d}}    
\renewcommand{\d}{\;\mathrm{d}} 
\renewcommand{\dim}{d}
\newcommand{\ubox}{\left]-1/2,1/2\right[^\dim}
\newcommand{\ux}{\underline{x}}
\newcommand\numberthis{\addtocounter{equation}{1}\tag{\theequation}}
\renewcommand{\epsilon}{\varepsilon}
\renewcommand{\phi}{\varphi}

\begin{document}

\title[Magnetic $\Psi$DO's seen as Hofstadter matrices]{Magnetic pseudodifferential operators represented as generalized Hofstadter-like matrices}

\author[H.D.~Cornean]{Horia D. Cornean}
\address[H.D.~Cornean]{Department of Mathematical Sciences, Aalborg University \\ Skjernvej 4A, 9220 Aalborg, Denmark.}
\email{cornean@math.aau.dk}

\author[H.~Garde]{Henrik Garde}
\address[H.~Garde]{Department of Mathematical Sciences, Aalborg University \\ Skjernvej 4A, 9220 Aalborg, Denmark.}
\email{henrik@math.aau.dk (corresponding author)}

\author[B.~St{\o}ttrup]{Benjamin St{\o}ttrup}
\address[B.~St{\o}ttrup]{Department of Mathematical Sciences, Aalborg University \\ Skjernvej 4A, 9220 Aalborg, Denmark.}
\email{benjamin@math.aau.dk}

\author[K.S.~S{\o}rensen]{Kasper S. S{\o}rensen}
\address[K.S.~S{\o}rensen]{Department of Mathematical Sciences, Aalborg University \\ Skjernvej 4A, 9220 Aalborg, Denmark.}
\email{kasper@math.aau.dk}

\begin{abstract}
First, we reconsider the magnetic pseudodifferential calculus and show that for a large class of non-decaying symbols, their corresponding magnetic pseudodifferential operators can be represented, up to a global gauge transform, as generalized Hofstadter-like, bounded matrices. As a by-product, we prove a Calder{\' o}n--Vaillancourt type result. Second, we make use of this matrix representation and prove sharp results on the spectrum location when the magnetic field strength $b$ varies. Namely, when the operators are self-adjoint, we show that their spectrum (as a set) is at least $1/2$-H\"older continuous with respect to $b$ in the Hausdorff distance. Third, when the magnetic perturbation comes from a constant magnetic field we show that their spectral edges are Lipschitz continuous in $b$. The same Lipschitz continuity holds true for spectral gap edges as long as the gaps do not close.
\end{abstract}

\maketitle

\KEY
magnetic pseudodifferential operators, 
spectral estimates,
generalized Hofstadter matrices.
\EndKEY

\MSC
47A10,  
47G30, 
47G10. 
\EndMSC

\section{Introduction and Main Results}

\subsection{The general setting}
Let $d\geq 2 $ and if $x\in \R^d$ we denote $\jn{x}\coloneqq (1+\abs{x}^2)^{1/2}$. Let 
$$BC^\infty(\R^d):=\left \{f\in C^\infty(\R^d;\R):\;\sup_{x\in \R^d}\vert \partial^\alpha f(x)\vert <\infty,\quad \forall \alpha\in \No^d\right \}.$$

We consider a magnetic field given by a $2$-form $B(x)=\sum_{i,j} B_{ij}(x) \,\di x_i \wedge \di x_j$ with $B_{ij}=-B_{ji}$, $B_{ij}\in BC^\infty(\R^d)$ and $ \partial_k B_{ij}+\partial_j B_{ki}+\partial_i B_{jk}=0 $, i.e.\ $\di B=0$. Since $B$ is closed, we may write $B=\di A$ for some (non unique) $1$-form $A$. We will only work with the so-called transverse gauge~\cite{CHP2018}, defined as follows: for every $x'\in \R^d$ let
\begin{align*}
A_j(x,x'):=-\sum_{k=1}^d \int_0^1 s(x_k-x'_k)B_{jk}(x'+s(x-x'))\d s,
\end{align*}
and observe that $B=\di A(\cdot,x')$ independently of $x'$. Let $\Gamma_{x,x'}$ denote the oriented segment linking $x'$ with $x$. The $1$-form  $A(\cdot,0)-A(\cdot,x')$ is closed  and
\begin{align*}
\phi(x,x'):=\int_{\Gamma_{x,x'}}A(\cdot,0)-A(\cdot,x')=\int_{\Gamma_{x,x'}} A(\cdot,0)
\end{align*}
satisfies 
\begin{align*}
\partial_{x_j} \phi(x,x')=A_j(x,0)-A_j(x,x').
\end{align*}

Using Stokes' theorem we see that $\phi(x,x')$ equals the magnetic flux through the oriented triangle having vertices at $0$, $x$ and $x'$. We now list three important properties of $\phi$. For all $x,x',y,z\in \R^d$ and $\alpha,\alpha',\beta\in \No^d$ we have:
\begin{enumerate}
	\item There exists a constant $C_{\alpha,\alpha'}$ such that
	\begin{align}\label{eq:phiineq}
	\abs{\partial_x^\alpha\partial_{x'}^{\alpha'}\phi(x,x')}\leq C_{\alpha,\alpha'}\abs{x}\abs{x'};
	\end{align}
	\item $\phi(x,x')=-\phi(x',x)$;
	\item If $\Delta(x,y,z)$ denotes the area of the triangle with vertices $x,y,z\in \R^d$ then the map $\fl\colon \R^{3d}\to \R$ given by
	\begin{align*}
	\fl(x,y,z)\coloneqq\phi(x,y)+\phi(y,z)-\phi(x,z)
	\end{align*}
	is the magnetic flux through the triangle with vertices $x,y,z$ and satisfies
	\begin{align}\label{eq:fluxestimate}
	\abs{\partial_{x}^{\alpha}\partial_{y}^{\alpha'} \fl(x,y,z)}\leq C_{\alpha,\alpha'}\Delta(x,y,z),
	\end{align}
	for some constant $C_{\alpha,\alpha'}$.
\end{enumerate}
Given such a $\phi$ we define the \emph{magnetic symbol class} $\Mphi$ to be the set of all functions on the form
\begin{align*}
a_b(x,x',\xi)=\e^{\I b \phi(x,x')}a(x,x',\xi),
\end{align*}
where $b\in \R$ and $a\in \Cinf(\R^{3d})$ is any function for which there exists $M\geq 0$ such that
\begin{align}\label{eq:symbol}
\abs{\partial_{x}^\alpha\partial_{x'}^{\alpha'}\partial_\xi^\beta a(x,x',\xi)}\leq C_{\alpha,\alpha',\beta}\jn{x-x'}^M,
\end{align}
for every $\alpha,\alpha',\beta\in \No^d$ and some constant $C_{\alpha,\alpha',\beta}$. Note that we allow a polynomial growth in the ``relative coordinate'' direction $x-x'$. We associate to each magnetic symbol $a_b\in \Mphi$ a \emph{magnetic pseudodifferential operator} $\Op(a_b)\colon \SR\to \SpR$ given by
\begin{align}\label{hc1}
\pair{\Op(a_b)f}{g}\coloneqq\frac{1}{(2\pi)^{d}}\int_{\R^{3d}} \e^{\I \xi\cdot(x-x')}\e^{\I b \phi(x,x')}a(x,x',\xi)f(x')\overline{g(x)}\d x' \d x \d \xi,
\end{align}
for $f,g\in \SR$. By~\eqref{eq:symbol} and~\eqref{eq:phiineq} it follows that $\Op(a_b)$ is well-defined. Note that this is not the usual magnetic Weyl quantisation procedure \cite{IMP07,IMP10}, which associates a H\"ormander symbol \cite{Ho1, Ho3} $\tilde{a}\in S_{0,0}^0(\R^{2d})$ to the following operator
\begin{align}\label{hc3}
	\pair{\Op_b^{\rm W}(\tilde{a})f}{g}\coloneqq\frac{1}{(2\pi)^{d}}\int_{\R^{3d}} \e^{\I \xi\cdot(x-x')}\e^{\I b \phi(x,x')}\tilde{a}((x+x')/2,\xi)f(x')\overline{g(x)}\d x' \d x \d \xi.	
\end{align}
In Theorem \ref{thm:main} we will show that $\Op(a_b)$ can be extended to a bounded operator on $L^2(\R^d)$, provided $a_b\in \Mphi$.
We immediately see that the magnetic Weyl operators belong to our class of magnetic pseudodifferential operators. On the other hand (see Remark \ref{remarkhc1} for more details), one can also prove that the opposite inclusion holds, in the sense that given one of "our" bounded operators one can construct via the magnetic Beals criterion \cite{CHP2018, IMP10} a magnetic Weyl symbol which generates the same operator. Nevertheless, working with our class seems to be more convenient when one shows that certain commutators can be extended to bounded operators on $L^2(\R^d)$.

The first goal of our paper is to show that, up to a global unitary gauge transformation, any such object can be identified with a \emph{bounded} generalized matrix acting on $\ell^2(\Z^d;L^2(\Omega))$ where $\Omega:=\,\,]-1/2,1/2[^d$ is the open unit $d$-hypercube.

The second goal is to study how their spectrum varies with $b$ (as a set) when the operators are self-adjoint.

\subsection{Recent developments} Magnetic Schr\"odinger operators of the type 
$H_b:=\sum_{j=1}^d(-\I\partial_{x_j}-bA_j)^2+V$ where $V$ is a scalar potential play a central role in both atomic and solid-state physics. When the magnetic field is long-range (i.e.\ it does not  decay fast enough at infinity), the corresponding magnetic potentials are no longer bounded perturbations and the spectral analysis is more involved. 

There is a substantial amount of literature dedicated to such operators, especially on the problem of obtaining effective magnetic Hamiltonians. From the physics literature we only mention the pioneering works of Peierls \cite{Pe} and Luttinger \cite{Lu}. The mathematical community became interested in the problem during the Eighties and gradually put it on a firm mathematical foundation. The works by Nenciu \cite{Ne-RMP}, and Helffer and Sj\"ostrand \cite{HS1, HS, S} were probably the first ones where the existence of magnetic tight-binding models was rigorously established. Nenciu  \cite{Ne-JMP} then showed that the resolvent $(H_b-z)^{-1}$ can be seen as a twisted magnetic integral operator and that the singular behaviour comes from a phase factor like $\e^{\I b\varphi(x,x')}$. 

Moreover, it was observed \cite{KO, MP} that in the presence of a non-constant magnetic field, the usual Weyl pseudodifferential calculus based on the minimal coupling principle at the level of classical symbols does not lead to gauge invariant formulas. Iftimie, M{\u a}ntoiu and Purice  \cite{IMP07, IMP10, IP11, IP14} introduced the so-called magnetic Weyl pseudodifferential calculus in which they treated operators like in \eqref{hc3}. The case $m=\rho=\delta=0$ was inherently more difficult, but in \cite{IMP10} they managed to prove a magnetic version of the Calder{\' o}n-Vaillancourt theorem and they also generalized the Beals criterion \cite{Beals, Bo} to the magnetic case.  

Several aspects of spectral and scattering theory using magnetic Weyl pseudodifferential calculus were analysed in \cite{ML, MPR1}. Lein and De Nittis \cite{dNL}, Panati, Spohn and Teufel \cite{PST}, and Freund and Teufel \cite{FrTe} developed a pseudodifferential calculus adapted for magnetic Bloch systems and applied it to various problems coming from the space-adiabatic perturbation theory. 

A special class of results concerns the resolvent set stability of magnetic Schr\"odinger operators and the Hausdorff regularity of the spectrum when $b$ varies. Continuity of the spectrum can be proved under quite general conditions on the Hamiltonians \cite{AMP, BB, B}, while more refined properties like the Lipschitz behaviour of spectral edges were first proved by Bellissard \cite{B2} for discrete Hofstadter-like models \cite{Hof}. Cornean, Purice and Helffer \cite{Co, CHP1, CHP2, CP-1, CP-2} extended this to continuous magnetic Schr\"odinger operators, and the magnetic Weyl calculus played a crucial role.  

\subsection{Main results}
Recall that $\Omega=\ubox$ and define:
\begin{align*}
\sH:=\bigoplus_{\gamma\in \Z^d} L^2(\Omega)=\Big\{(f_\gamma)_{\gamma\in \Z^d}\subset L^2(\Omega) \mid \sum_{\gamma\in \Z^d} \norm{f_\gamma}_{L^2(\Omega)}^2 <\infty\Big\},
\end{align*}
which is a Hilbert space when equipped with the inner product
\begin{align*}
\ip{(f_\gamma)}{(g_\gamma)}_{\sH}\coloneqq\sum_{\gamma\in \Z^d} \ip{f_\gamma}{g_\gamma}_{L^2(\Omega)}.
\end{align*}
 Furthermore, for any $b\in \R$, let $U_b\colon L^2(\R^d)\to \sH$ be given by 
\begin{align}\label{hc2}
(U_bf)_\gamma(\cdot)\coloneqq \e^{-\I b \phi(\cdot+\gamma,\gamma)} \chi_\Omega(\cdot)f(\cdot+\gamma),
\end{align}
for all $f\in L^2(\R^d)$, where $\chi_\Omega$ denotes the characteristic function on $\Omega$. The operator $U_b$ is unitary and
\begin{align*}
[U^*_b(f_\gamma)_{\gamma\in\Z^d}](\cdot)=\sum_{\gamma\in\Z^d} \e^{\I b \phi(\cdot,\gamma)} \chi_\Omega(\cdot-\gamma)f_\gamma(\cdot-\gamma).
\end{align*}

We say that an operator $\cA$ on $\sH$ is a \emph{generalized matrix} of the operators $(\cA_{\gamma,\gamma'})_{\gamma,\gamma'\in\Z^d}\subset B(L^2(\Omega))$ when:
\begin{align*}
\cA= \{\cA_{\gamma,\gamma'}\}_{\gamma,\gamma'\in\Z^d}, \quad (\cA f)_\gamma=\sum_{\gamma'\in \Z^d} \cA_{\gamma,\gamma'}f_{\gamma'}
\end{align*}
for all $f=(f_\gamma)_{\gamma\in \Z^d}\in \sH$. One may also see that $\cA$ acts on $\ell^2(\Z^d;L^2(\Omega))$.

The Hausdorff distance between two compact sets $X,Y\subset \R$ is defined as:  
		\begin{align*}
		\dH(X,Y):=\max\{\sup_{x\in X} \dist(x,Y),\sup_{y\in Y} \dist(y,X)\}.
		\end{align*}

We are now ready to state our main theorem.
\begin{theorem}\label{thm:main}
	If $a_b\in \Mphi$ with $b\in [0,b_{\rm max}]$ for some $b_{\rm max}>0$, then: 
	\begin{enumerate}
		\item The operator $\Op(a_b)$ in \eqref{hc1} extends to a bounded operator on $L^2(\R^d)$ and for each $\gamma,\gamma'\in \Z^d$ there exists $\cA_{\gamma\gamma',b}\in B(L^2(\Omega))$ such that (see \eqref{hc2})
		\begin{align}\label{eq:mainthm1}
		U_b\Op(a_b)U_b^*=\{\e^{\I b \phi(\gamma,\gamma')} \cA_{\gamma\gamma',b}\}_{\gamma,\gamma'\in \Z^d}.
		\end{align}
		Moreover, for every $N\in \N$ there exists a constant $C_N$ such that
		\begin{align}\label{eq:Aggbdecay}
		\norm{\cA_{\gamma\gamma',b}}\leq C_N \jn{\gamma-\gamma'}^{-N},
		\end{align}
		and
		\begin{align}\label{eq:AggbLipschitz}
		\norm{\cA_{\gamma\gamma',b}-\cA_{\gamma\gamma',b'}}\leq C_N\jn{\gamma-\gamma'}^{-N}\abs{b-b'},\quad \textup{for } b,b'\in [0,b_{\rm max}],
		\end{align}
		for all $\gamma,\gamma'\in \Z^d$.
	\end{enumerate}
	Additionally, if $a(x,x',\xi)=\overline{a(x',x,\xi)}$ then $\Op(a_b)$ is self-adjoint and in this case:
	\begin{enumerate}
		\setcounter{enumi}{1}
		\item The spectrum of $\Op(a_b)$ is $\frac{1}{2}$-Hölder continuous in $b$ on the interval $[0,b_{\rm max}]$, i.e.\ there exists a constant $C$ such that
		\begin{align}\label{eq:main3}
		\dH(\sigma(\Op(a_b)),\sigma(\Op(a_{b'})))\leq C\abs{b-b'}^{1/2},
		\end{align}
		for all $b,b'\in [0,b_{\rm max}]$.
		
		\item Assume that $\phi$ comes from a constant magnetic field, i.e.\ $\phi(x,x')=\frac{1}{2}x^\top Bx'$ where $B$ is an antisymmetric matrix. If $E_b$ denotes the maximum (minimum) of $\sigma(\Op(a_b))$, then it is Lipschitz continuous in $b$ on $[0,b_{\rm max}]$. Furthermore, if $e_b$ denotes an edge of a spectral gap which remains open when $b$ varies in some interval $[b_1,b_2]\subset [0,b_{\rm max}]$, then $e_b$ is Lipschitz continuous on $[b_1,b_2]$.
	\end{enumerate}
\end{theorem}

\begin{remark}
The representation \eqref{eq:mainthm1} justifies the name ``generalized Hofstadter matrix'' \cite{Hof, B2}. In the ``classical'' Hofstadter-like setting one deals with a discrete operator acting on $\ell^2(\Z^d;\C)$ where the matrix entries are complex numbers. In our case they are bounded operators on $L^2(\Omega)$. Furthermore, the matrix elements are strongly localized around the diagonal as in \eqref{eq:Aggbdecay}. We also note that after rotating $\Op(a_b)$ with $U_b$, the only singular behaviour in $b$ is left in the ``Peierls''-like phase $\e^{\I b\varphi(\gamma,\gamma')}$, since the entries $\cA_{\gamma\gamma',b}$ are Lipschitz in $b$ in the norm topology, see \eqref{eq:AggbLipschitz}. For  nearest-neighbor Hofstadter-like operators it is known from the works of Bellissard, Helffer--Sj\"ostrand and Nenciu that the spectrum is $\frac{1}{2}$-H\"older continuous and that the exponent $\frac{1}{2}$ is optimal in the sense that gaps of order $|b-b'|^{1/2}$ may open in the spectrum (for more details see \cite{Co,Ne05} and references within).
\end{remark}
\begin{remark}\label{remarkhc1} Our class $M_\varphi(\R^{3d})$ of symbols  which obey \eqref{eq:symbol} is more convenient to work with, but \emph{it does not} generate ``more'' operators than the ``usual'' magnetic Weyl quantisation \eqref{hc3}. Let us show that given any operator $\Op(a_b)$ as in \eqref{hc1} one may find  a H\"ormander symbol $\tilde{a}\in S^0_{0,0}(\R^{2d})$ such that $\Op(a_b)=\Op_b^{\rm W}(\tilde{a})$, where $\Op_b^{\rm W}(\tilde{a})$ is as in~\eqref{hc3}. In order to prove this we use the Beals criterion for magnetic pseudodifferential operators~\cite{IMP10, CHP2018}.  Namely, let us denote $ W_k=X_k $ if $k=1,2,\dots,d$ and $W_k=-\I\partial_{x_{k-d}}-bA_{k-d}(\cdot,0)$ if $ k=d+1,\dots,2d$. Then we will show using Theorem \ref{thm:main}(1) that all the commutators of the form
	\begin{align*}
	[W_{j_1},[W_{j_2},\dots,[W_{j_m},\Op(a_b)]\ldots]],
	\end{align*}
	$j_\ell\in \{1,2,\dots,2d\}$, $ m\geq 1 $, can be extended to bounded operators on $L^2(\R^d)$, hence \eqref{hc3} holds due to the magnetic Beals criterion. We only show this for $m=1$, the general case follows by induction. 
	
	Indeed, integration by parts gives
	\begin{align*}
	\pair{[X_k,\Op(a_b)]f}{g}=\frac{\I}{(2\pi)^{d}}\int_{\R^{3d}} \e^{\I \xi\cdot(x-x')}\e^{\I b\phi(x,x')}(\partial_{\xi_k}a)(x,x',\xi)f(x')\overline{g(x)}\d x' \d x \d \xi,
	\end{align*}
	which by Theorem~\ref{thm:main}(1) can be extended to a bounded operator on $L^2(\R^d)$. Using again integration by parts together with the fact that $\partial_{x_j} \phi(x,x')=A_j(x,0)-A_j(x,x')$ we obtain after a straightforward computation that the commutator $[(-\I\partial_{x_j}-bA_j),\Op(a_b)]$ is a magnetic pseudodifferential operator with magnetic symbol 
	\begin{align*}
	\e^{\I b\phi(x,x')}\big (bA_j(x',x)-bA_j(x,x')-\I(\partial_{x_j}+\partial_{x'_j})\big )a(x,x',\xi)\in M_\varphi(\R^{3d}).
	\end{align*}
	Here we see the advantage of allowing polynomial growth in $x-x'$, because even though $A_j(x,x')$ and $A_j(x',x)$ have a linear growth in $|x-x'|$ we can directly apply Theorem~\ref{thm:main}(1) and the commutator can be extended to a bounded operator on $L^2(\R^d)$. 
\end{remark}

\begin{remark}
	If it is possible to choose a vector potential $A$ such that 
	\begin{align}\label{eq:vecpotben1}
	\abs{\partial_{x}^\alpha A_j(x)}\leq C_{\alpha},
	\end{align}
	for all multiindices $\alpha$ with $\abs{\alpha}>0$, then every magnetic pseudodifferential operator would correspond to a non-magnetic Weyl type pseudodifferential operator. In order to show this we use the non-magnetic Beals criterion. First, note that the commutator $[-bA_j(\cdot),\Op(a_b)]$ is a magnetic pseudodifferential operator with magnetic symbol 
	\begin{align*}
	\e^{\I b\phi(x,x')}b(A_j(x')-A_j(x))a(x,x',\xi).
	\end{align*}
	By Theorem~\ref{thm:main}(1) the above commutator extends to a bounded operator in $L^2(\R^d)$. Using Remark~\ref{remarkhc1} we obtain that $[-\I\partial_{x_j},\Op(a_b)]$ extends to a bounded operator on $L^2(\R^d)$. After an induction argument we obtain that $\Op(a_b)$ satisfies the classical non-magnetic Beals criterion. 
	
	However, we note that \eqref{eq:vecpotben1} does not necessarily hold for the transverse gauge, although the constant magnetic field obeys this condition. Furthermore, to obtain sharp results on the behaviour of $\sigma(\Op(a_b))$ as the magnetic field strength varies, using the non-magnetic Weyl quantisation is not convenient when one  works with nonconstant magnetic fields.
\end{remark}

\subsection{The structure of the paper}

After this introduction, in  Section~\ref{sec:proof1} we prove  Theorem~\ref{thm:main}(1) by regularizing our magnetic symbol and writing the corresponding magnetic pseudodifferential operator as an integral operator with a smooth integral kernel. By rewriting in a clever way the kernel of this operator we are able to construct the right hand side of~\eqref{eq:mainthm1} as a strong limit of a regularized sequence of operators. 

In Section~\ref{sec:proof2} we prove Theorem~\ref{thm:main}(2) by adapting some ideas coming from geometric perturbation theory and \cite{CP-1}.   

In Section~\ref{sec:proof3} we prove Theorem~\ref{thm:main}(3) in the case when $E_b$ is the maximum of the spectrum. Finally, we show how to deal with inner gap edges. 

\vspace{0.5cm}

\noindent {\bf Acknowledgments.} H.C. gratefully acknowledges inspiring discussions with S.~Beckus, J.~Bellissard, B.~Helffer, G.~Nenciu, and R.~Purice. 

This research is supported by grant 8021--00084B \emph{Mathematical Analysis of Effective Models and Critical Phenomena in Quantum Transport} from The Danish Council for Independent Research \textbar\ Natural Sciences.

\section{Proof of Theorem~\ref{thm:main}(1) } \label{sec:proof1}
For simplicity we assume that $a_b(x,x',\xi)=\e^{\I b \phi(x,x')}a(x,x',\xi)$ where $a$ is a symbol of Hörmander class $S^0_{0,0}(\R^{3d})$ i.e.\ $M=0$ in~\eqref{eq:symbol}. The proof can then be extended to any $M\geq 0$ (see Remark~\ref{rem:1<<<<<<<<<<<<<<<<} for more details).
\subsection{Regularization of Magnetic Symbols}
We begin by regularizing the symbol $a_b$ in order to write the corresponding magnetic pseudodifferential operator as a generalized matrix of integral operators with smooth integral kernels.
\begin{lemma}
	Let $a_b\in \Mphi$. For $\epsilon>0$ define $a_{b,\epsilon}\colon \R^{3d}\to \C$ by
	\begin{align*}
	a_{b,\epsilon}(x,x',\xi)\coloneqq a_b(x,x',\xi)\e^{-\epsilon \jn{\xi}}
	\end{align*}
	and $K_{b,\epsilon}\colon \R^{2d}\to \C$ by
	\begin{align*}
	K_{b,\epsilon}(x,x')\coloneqq \frac{1}{(2 \pi)^d}\int_{\R^d}\e^{\I \xi\cdot(x-x')}a_{b,\epsilon}(x,x',\xi)\d \xi.
	\end{align*}
	Then the integral operator with kernel $K_{b,\epsilon}$ is a bounded operator on $L^2(\R^d)$ and  for $f\in \SR$ we have
	\begin{align}\label{eq:abeps}
	(\Op(a_{b,\epsilon})f)(x)=\int_{\R^d}K_{b,\epsilon}(x,x')f(x')\d x'.
	\end{align}
\end{lemma}
\begin{proof}
	The proof is a consequence of integration by parts, Schur-Holmgren lemma~\cite[Lemma 18.1.12]{Ho3} and the identity
	\begin{align}\label{eq:jnrewrite}
	\jn{x}^{2n}=\sum_{\abs{\alpha}\leq n} C_\alpha x^{2\alpha},
	\end{align}
	which holds for $x\in \R^d$.
Using Fubini's theorem gives
\begin{align*}
\pair{\Op(a_{b,\epsilon})f}{g}=\int_{\R^{2d}}K_{b,\epsilon}(x,x')f(x') \overline{g(x)}\d x'\d x
\end{align*}
for $f,g\in \SR$ which proves~\eqref{eq:abeps}.
\end{proof}

Next we show that the operator $\cA_{b,\epsilon}\coloneqq U_b\Op(a_{b,\epsilon})U_b^*$ can be written as a generalized matrix of integral operators on $L^2(\Omega)$. In the following we underline variables to indicate that they belong to $\Omega$. By the definition of $U_b, U_b^*$ and~\eqref{eq:abeps} we have that
\begin{align}\label{eq:abepsben1}
(\cA_{b,\epsilon}(f_{\gamma'}))_\gamma(\ux)=\sum_{\gamma'\in \Z^d}\int_\Omega K_{b,\epsilon}(\ux+\gamma,\ux'+\gamma')\e^{\I b(\phi(\ux'+\gamma',\gamma')-\phi(\ux+\gamma,\gamma))}f_{\gamma'}(\ux')\d \ux',
\end{align}
for $(f_{\gamma'})\in \sH$. If for every $\gamma,\gamma'\in \Z^d$ we define
\begin{align*}
\fl_{\gamma,\gamma'}(x,x')\coloneqq \fl(x+\gamma,\gamma',\gamma)+\fl(x+\gamma,x'+\gamma',\gamma')
\end{align*}
and 
\begin{align}\label{eq:ubopaubstar}
K_{\gamma,\gamma'}(\ux,\ux'):=\frac{1}{(2\pi)^d}\int_{\R^d}  \e^{\I \xi\cdot(\ux+\gamma-\ux'-\gamma')}\e^{\I b \fl_{\gamma,\gamma'}(\ux,\ux')} \e^{-\epsilon \jn{\xi}} a(\ux+\gamma,\ux'+\gamma',\xi) \d \xi,
\end{align}
then we can use the identity
\begin{align*}
\phi(x+\gamma,x'+\gamma')=-\phi(x'+\gamma',\gamma')+\phi(x+\gamma,\gamma)+\phi(\gamma,\gamma')+\fl_{\gamma,\gamma'}(x,x')
\end{align*}
to write \eqref{eq:abepsben1} as
\begin{align*}
(\cA_{b,\epsilon}(f_{\gamma'}))_\gamma(\ux)=\sum_{\gamma'\in \Z^d}\e^{\I b\phi(\gamma,\gamma')}\int_\Omega K_{\gamma,\gamma'}(\ux,\ux') f_{\gamma'}(\ux')\d \ux'.
\end{align*}
This shows that the operator $\cA_{b,\epsilon}$ is a generalized matrix i.e.\
\begin{align}\label{eq:AbepsdirectsumBen}
\cA_{b,\epsilon}= \{\e^{\I b\phi(\gamma,\gamma')}\cA_{\gamma\gamma',b,\epsilon}\}_{\gamma,\gamma'\in \Z^d},
\end{align}
where the operators $\cA_{\gamma\gamma',b,\epsilon}$ are integral operators with kernel $K_{\gamma,\gamma'}$. The next step in the proof is to construct operators $\cA_{\gamma\gamma',b}$, which are strong limits of $\cA_{\gamma\gamma',b,\epsilon}$ as $\epsilon\to 0$.
\subsection{Construction of \texorpdfstring{$\cA_{\gamma\gamma',b}$}{Aggb}}
We rewrite the kernel of the operator $\cA_{\gamma\gamma',b,\epsilon}$ for each $\gamma,\gamma'\in \Z^d$ in a way that allows us to take $\epsilon$ to zero. Before we construct the operators $ \cA_{\gamma\gamma',b} $ we note, as a consequence of~\eqref{eq:fluxestimate}, that for every $\alpha,\alpha'\in \No^d$ there exists $C_{\alpha,\alpha'}$ such that
\begin{align}\label{eq:fluxbound}
\abs{\partial_x^\alpha\partial_{x'}^{\alpha'}\fl_{\gamma,\gamma'}(x,x')}\leq C_{\alpha,\alpha'}\jn{\gamma-\gamma'}
\end{align}
for all $x,x'\in \tilde{\Omega}\coloneqq[-\pi,\pi]^d$.

The first step in the construction is to obtain a Fourier series (for each fixed $\xi$) of the function
\begin{align*}
\Omega^2\ni(\ux,\ux')\mapsto a(\ux+\gamma,\ux'+\gamma',\xi)\e^{\I b\fl_{\gamma,\gamma'}(\ux,\ux')},
\end{align*}
for all $\gamma,\gamma'\in \Z^d$. In order to circumvent the problem that this function is not necessarily periodic let $g\in C^\infty_0(\tilde{\Omega})$ be such that $0\leq g\leq 1$ and $g\equiv 1$ on some open set containing $\Omega$. Then for every $\gamma,\gamma' \in \Z^d$ the function
\begin{align}\label{eq:atwitle}
\tilde{\Omega}^2\ni(x,x')\mapsto g(x)g(x')a(x+\gamma,x'+\gamma',\xi)\e^{\I b \fl_{\gamma,\gamma'}(x,x')},
\end{align}
can be extended to a periodic function in $x,x'$ and hence has a Fourier series expansion. Before we consider this expansion we note that for any $\alpha,\alpha',\beta\in \No^d$ Leibniz's rule and~\eqref{eq:fluxbound} gives the existence of a constant $C_{\alpha,\alpha',\beta}$, not depending on $b$, satisfying
\begin{align}\label{eq:agammagammaprimebbound}
\babss{\partial_{x}^\alpha\partial_{x'}^{\alpha'}\partial_{\xi}^{\beta}\Big(g(x)g(x')a(x+\gamma,x'+\gamma',\xi)\e^{\I b \fl_{\gamma,\gamma'}(x,x')}\Big)}&\leq C_{\alpha,\alpha',\beta}\jn{\gamma-\gamma'}^{\abs{\alpha}+\abs{\alpha'}}.
\end{align}
This is because the left hand side depends polynomially on $b$, therefore by the assumption that $b\in [0,b_{\rm max}]$ it follows that the right hand side can be chosen independently of $b$.

We would like to obtain an explicit decay in the summation variables $m,m'$ for the Fourier series of~\eqref{eq:atwitle}. To avoid cumbersome notation we will annotate functions and operators, within this section, which depend on the variables $\gamma,\gamma',m,m'\in \Z^d$ with a tilde accent. To obtain the aforementioned decay in the Fourier series we define for every $\gamma,\gamma',m,m'\in \Z^d$ the function
\begin{align*}
\tilde{a}_{b}(\xi)\coloneqq\frac{(\jn{m}\jn{m'})^{2d}}{(2\pi)^{2d}}\int_{\tilde{\Omega}^2} \e^{-\I(m\cdot x+m'\cdot x')}g(x)g(x')a(x+\gamma,x'+\gamma',\xi)\e^{\I b \fl_{\gamma,\gamma'}(x,x')}\d x \d x',
\end{align*}
and use integration by parts together with \eqref{eq:agammagammaprimebbound} to obtain the estimate 
\begin{align}\label{eq:tildeagammagammaprimemmprimeNestimate}
\abs{\partial_{\xi}^{\beta}\tilde{a}_{b}(\xi)}&\leq C_{\beta}\jn{\gamma-\gamma'}^{4d},
\end{align}
for all $\beta\in \No^{d}$. The Fourier series of the function in \eqref{eq:atwitle} then becomes
\begin{align*}
g(x)g(x')a(x+\gamma,x'+\gamma',\xi)\e^{\I b \fl_{\gamma,\gamma'}(x,x')}=\sum_{m,m'\in \Z^d} \frac{\e^{\I(m\cdot x+m'\cdot x')}}{(\jn{m}\jn{m'})^{2d}}\tilde{a}_{b}(\xi).
\end{align*}
Since $g \equiv 1$ on $\Omega$ it follows that the kernels $K_{\gamma,\gamma'}$ in \eqref{eq:ubopaubstar} can be written as
\begin{align*}
K_{\gamma,\gamma'}(\ux,\ux')&=\frac{1}{(2\pi)^d}\sum_{m,m'\in \Z^d}\frac{1}{(\jn{m}\jn{m'})^{2d}} \int_{\R^d}  \e^{\I \xi\cdot(\ux+\gamma-\ux'-\gamma')} \e^{\I(m\cdot\ux+m'\cdot\ux')}
\tilde{a}_{b}(\xi) \e^{-\epsilon\jn{\xi}}\d \xi.
\end{align*}
Since the function $\tilde{a}_b$ only depends on $\xi$ we can use the exponential factors $\e^{\I \xi\cdot \ux}$ and $\e^{\I \xi \cdot\ux'}$ that appear in $K_{\gamma,\gamma'}$ to write each $\cA_{\gamma\gamma',b,\epsilon}$ as a series of pseudodifferential operators. Specifically, if we for every $\gamma,\gamma',m,m'\in \Z^d$ define the operators $\tilde{\cA}_{b,\epsilon}\colon C^\infty_0(\Omega)\to \SR$ by 
\begin{equation*}
(\tilde{\cA}_{b,\epsilon}h)(x)\coloneqq \e^{\I m\cdot x}\sF^{-1}\bigg[\e^{\I (*)\cdot (\gamma-\gamma')}\tilde{a}_b(*) \e^{-\epsilon\jn{*}}\sF\Big( \e^{\I m'\cdot(\cdot)}h(\cdot)\Big)(*)\bigg](x),
\end{equation*}
for all $\epsilon\geq 0$, then Fubini's theorem implies that
\begin{equation}\label{eq:AggbepsBen}
(\cA_{\gamma\gamma',b,\epsilon}h)(\ux)=\sum_{m,m'\in \Z^d} \frac{1}{(\jn{m}\jn{m'})^{2d}}(\tilde{\cA}_{b,\epsilon} h)(\ux),
\end{equation}
for all $h\in C^\infty_0(\Omega)$ and $\epsilon > 0$.
Since $\tilde{\cA}_{b,\epsilon}$ is well-defined even when $\epsilon=0$ we define $\cA_{\gamma\gamma',b}$ on $C^\infty_0(\Omega)$ by
\begin{equation}\label{eq:AggbBen}
(\cA_{\gamma\gamma',b}h)(\ux)\coloneqq(\cA_{\gamma\gamma'b,0}h)(\ux)=\sum_{m,m'\in \Z^d} \frac{1}{(\jn{m}\jn{m'})^{2d}}(\tilde{\cA}_{b,0}h)(\ux).
\end{equation}
We will later prove that $\cA_{\gamma\gamma',b,\epsilon}$ converges strongly to $\cA_{\gamma\gamma',b}$ and use this to show that $\cA_{\gamma\gamma',b}$ satisfy Theorem~\ref{thm:main}. 

\subsection{Norm Estimates: Proof of \texorpdfstring{\eqref{eq:Aggbdecay}}{(1.5)} and \texorpdfstring{\eqref{eq:AggbLipschitz}}{(1.6)}}
The aim of this section is to prove the following lemma, from which both~\eqref{eq:Aggbdecay} and~\eqref{eq:AggbLipschitz} follow immediately.
\begin{lemma}\label{lem:3}
	Suppose $b,b'\in [0,b_{\rm max}]$. Then for every $N\in \N$ there exists a constant $C_{N}$ such that
	\begin{align}\label{eq:prop1}
	\norm{\cA_{\gamma\gamma',b,\epsilon}h}_{L^2(\Omega)}\leq C_{N}\jn{\gamma-\gamma'}^{4d-2N}\norm{h}_{L^2(\Omega)}
	\end{align}
	and
	\begin{align}\label{eq:prop2}
	\norm{(\cA_{\gamma\gamma',b,\epsilon}-\cA_{\gamma\gamma',b',\epsilon}) h}_{L^2(\Omega)}&\leq C_{N} \jn{\gamma-\gamma'}^{4d+1-2N}\abs{b-b'}\norm{h}_{L^2(\Omega)},
	\end{align}
	for all $h\in C^\infty_0(\Omega)$ and all $\epsilon\in [0,1]$. 
\end{lemma}
From Lemma~\ref{lem:3} it follows that $\cA_{\gamma\gamma',b}$ extends to a bounded operator on $L^2(\Omega)$.

\begin{proof}
Let $N\in \N$ be arbitrary. From \eqref{eq:AggbepsBen} and \eqref{eq:AggbBen} it is clear that in order to estimate \eqref{eq:prop1} we have to estimate the norm of $\jn{\gamma-\gamma'}^{2N} \tilde{\cA}_{b,\epsilon}$ for $\epsilon\in [0,1]$. Applying~\eqref{eq:jnrewrite} together with integration by parts and Leibniz's rule gives the existence of a constant $M_N\in \N$ and sequences $(C_n)_{n=1}^{M_N}\subset \C$, $(\alpha_n)_{n=1}^{M_N}$, $(\alpha_n')_{n=1}^{M_N}$, $(\beta_n)_{n=1}^{M_N}\subset \No^d$ not depending on $h$ such that
\begin{align*}
\jn{\gamma-\gamma'}^{2N}(\tilde{\cA}_{b,\epsilon} h)(\ux)=\sum_{n=1}^{M_N} C_n \ux^{\alpha_n}\sF^{-1}\bigg[\e^{\I (*)\cdot (\gamma-\gamma')}\partial_{(*)}^{\beta_n}[\tilde{a}_b(*) \e^{-\epsilon\jn{*}}]\sF\Big((\cdot)^{\alpha_n'} \e^{\I m'\cdot(\cdot)}h(\cdot)\Big)(*)\bigg](\ux),
\end{align*}
for all $h\in C^\infty_0(\Omega)$ and $\epsilon\geq 0$.  

In order to show~\eqref{eq:prop1} it only remains to obtain a suitable estimate of the norm of the right hand side. By applying Parseval's identity twice we obtain
\begin{align*}
\jn{\gamma-\gamma'}^{2N}\norm{\tilde{\cA}_{b,\epsilon}h}_{L^2(\Omega)}\leq\sum_{n=1}^{M_N} C_n\norm{\partial_{(*)}^{\beta_n}[\tilde{a}_b(*)\e^{-\epsilon\jn{*}}]}_{L^\infty(\R^d)}\bigg(\int_{\Omega}\abs{(x')^{\alpha_n'} h(x')}^2\d x'\bigg)^{1/2},
\end{align*}
for all $h\in C^\infty_0(\Omega)$ and $\epsilon\geq 0$. Since $h\in C^\infty_0(\Omega)$ we have the bound $\abs{(x')^{\alpha_n'}h(x')}\leq \abs{h(x')}$ for all $n=1,\dots,M_N$. Combining this inequality with the estimate~\eqref{eq:tildeagammagammaprimemmprimeNestimate} and the fact that any finite number of derivatives of $\e^{-\epsilon\jn{\cdot}}$ is uniformly bounded for $\epsilon \in [0,1]$ gives the estimate 
\begin{align*}
\jn{\gamma-\gamma'}^{2N}\norm{\tilde{\cA}_{b,\epsilon}h}_{L^2(\Omega)}\leq C_{N}\jn{\gamma-\gamma'}^{4d}\norm{h}_{L^2(\Omega)},
\end{align*}
for all $h\in C^\infty_0(\Omega)$, $\epsilon\geq 0$ and some constant $C_{N}$ not depending on $b$. This completes the proof of~\eqref{eq:prop1}.

To prove \eqref{eq:prop2} we need to subtract two functions as in \eqref{eq:atwitle} but with different choices of $b$ and obtain an estimate similar to \eqref{eq:agammagammaprimebbound}. By \eqref{eq:atwitle} such a difference is given by
\begin{align*}
g(x)g(x')a(x+\gamma,x'+\gamma',\xi)\e^{\I b' \fl_{\gamma,\gamma'}(x,x')}[\e^{\I(b-b')\fl_{\gamma,\gamma'}(x,x')}-1],
\end{align*}
for $b,b'\in [0,b_{\rm max}]$. Using that for all $y\in \R$ we have
\begin{align}\label{eq:complexexponentialinequality}
\abs{\e^{\I y}-1}\leq \abs{y},
\end{align}
together with~\eqref{eq:fluxbound},~\eqref{eq:agammagammaprimebbound} gives for any $\alpha,\alpha',\beta\in \No^d$ the existence of a constant $C_{\alpha,\alpha',\beta}$ such that
\begin{align*}\label{eq:tildeagammagammaprimeboneminusbtwo}
\babss{\partial_{x}^\alpha\partial_{x'}^{\alpha'}\partial_{\xi}^{\beta}\Big(g(x)g(x')a(x+\gamma,x'+\gamma',\xi)\e^{\I b' \fl_{\gamma,\gamma'}(x,x')}&[\e^{\I(b-b')\fl_{\gamma,\gamma'}(x,x')}-1]\Big)}\\
&\qquad\leq C_{\alpha,\alpha',\beta}\abs{b-b'}\jn{\gamma-\gamma'}^{\abs{\alpha}+\abs{\alpha'}+1}.\numberthis
\end{align*}
Note that when we use Leibniz's rule on the left hand side every term will contain a factor on the form $(b-b')^n$ with $n\in \N$ and since $b,b'\in [0,b_{\rm max}]$ we can absorb the extra factors in the constant. By using~\eqref{eq:tildeagammagammaprimeboneminusbtwo} in calculations similar to those that gave~\eqref{eq:tildeagammagammaprimemmprimeNestimate} we obtain
\begin{align*}
\abs{\partial_{\xi}^{\beta}\tilde{a}_b(\xi)-\partial_{\xi}^{\beta}\tilde{a}_{b'}(\xi)}&\leq C_{\beta}\abs{b-b'}\jn{\gamma-\gamma'}^{4d+1},
\end{align*}
for all $\beta\in \No^d$ and somce constant $C_\beta$. With this estimate the proof of~\eqref{eq:prop2} follows the same way as the proof of~\eqref{eq:prop1}.
\end{proof}

\subsection{Strong convergence of \texorpdfstring{$\cA_{b,\epsilon}$}{Abe}}
In this section we prove that $\cA_{\gamma\gamma',b,\epsilon}$ converges strongly to $\cA_{\gamma\gamma',b}$ as $\epsilon$ goes to zero (cf.\ \eqref{eq:AggbBen}). Furthermore, we construct an operator $H_b$ as the generalized matrix with entries $\e^{\I b\phi(\gamma,\gamma')}\cA_{\gamma\gamma',b}$. Using the strong convergence $\cA_{\gamma\gamma',b,\epsilon}\to \cA_{\gamma\gamma',b}$ we prove that $\cA_{b,\epsilon}$ in \eqref{eq:AbepsdirectsumBen} converges strongly to $H_b$. Finally, we apply this to continuously extend $\Op(a_b)$ to an operator in $B(L^2(\R^d))$. 

\begin{lemma}\label{lem:ben1}
	For each $\gamma,\gamma'\in \Z^d$ the operators $\cA_{\gamma\gamma',b,\epsilon}$ converge strongly to $\cA_{\gamma\gamma',b}$ on $C^\infty_0(\Omega)$. 
\end{lemma}
\begin{proof}
Suppose that $h\in C^\infty_0(\Omega)$. From \eqref{eq:AggbepsBen} and \eqref{eq:AggbBen} it suffices to consider the operators $\tilde{\cA}_{b,\epsilon}-\tilde{\cA}_{b,0}$ for all $\gamma,\gamma',m,m'\in \Z^d$. Applying Parseval's identity once gives
\begin{align*}
\norm{(\cA_{\gamma\gamma',b,\epsilon}-\cA_{\gamma\gamma',b})h}_{L^2(\Omega)}\leq\sum_{m,m'\in\Z^d} \frac{1}{(\jn{m}\jn{m'})^{2d}} \norm{\tilde{a}_b(*) (\e^{-\epsilon\jn{*}}-1)\sF\Big(\e^{\I m'\cdot(\cdot)}h(\cdot)\Big)(*)}_{L^2(\Omega)}.
\end{align*}
Using Parseval's identity again shows that the $L^2$-norm appearing on the right hand side is bounded by a constant which is independent of $m$, $m'$ and $\epsilon$. Therefore it is enough to prove that this norm goes to $0$ with $\epsilon$ for a fixed $m$ and $m'$, which follows by an application of Lebesgue's dominated convergence theorem.
\end{proof}

To construct the operators $H_b$ we need the following general lemma on generalized matrices of operators.
\begin{lemma}\label{lem:directsumofoperatorsbounded}
	Suppose that there exists a constant $C$ and operators $(T_{\gamma,\gamma'})_{\gamma,\gamma'\in \Z^d}\subset B(L^2(\Omega))$ such that
	\begin{align*}
	\norm{T_{\gamma,\gamma'}f}_{L^2(\Omega)}\leq \frac{C\norm{f}_{L^2(\Omega)}}{\jn{\gamma-\gamma'}^{2d}},
	\end{align*}
	for every $\gamma,\gamma'\in \Z^d$ and $f\in C^\infty_0(\Omega)$. Then $T=\{T_{\gamma,\gamma'}\}_{\gamma,\gamma'\in \Z^d}$ is a bounded operator on $\sH$ with
	\begin{align*}
	\norm{T}\leq \sum_{\gamma\in \Z^d} \frac{C}{\jn{\gamma}^{2d}}. 
	\end{align*}
\end{lemma}
\begin{proof}
Let $f \in \{(f_\gamma) \in \sH \mid f_\gamma \in C^\infty_0(\Omega)\}$ and $S \colon \ell^2(\Z^d) \to \ell^2(\Z^d)$ an operator with matrix elements
\begin{align*}
S_{\gamma,\gamma'} = \frac{C}{\jn{\gamma-\gamma'}^{2d}}.
\end{align*}
Using a Schur-Holmgren estimate we get that $S$ is bounded and $\norm{S}\leq \sum_{\gamma\in \Z^d} \frac{C}{\jn{\gamma}^{2d}}$. Then:
\begin{align*}
\norm{Tf}_\sH^2 = \sum_{\gamma \in \Z^d} \Bnorm{\sum_{\gamma' \in \Z^d} T_{\gamma,\gamma'}f_{\gamma'}}_{L^2(\Omega)}^2 \leq \sum_{\gamma \in \Z^d} \Big(\sum_{\gamma' \in \Z^d} S_{\gamma,\gamma'}\norm{f_{\gamma'}}_{L^2(\Omega)}  \Big)^2 \leq \norm{S}^2\norm{f}_{\sH}^2.
\end{align*}
Since $T$ is linear and bounded on a dense set, it can be extended to the whole space $\sH$.
\end{proof}

By~\eqref{eq:Aggbdecay} and Lemma~\ref{lem:directsumofoperatorsbounded} we obtain that
\begin{align*}
H_b\coloneqq \{\e^{\I b \phi(\gamma,\gamma')} \cA_{\gamma\gamma',b}\}_{\gamma,\gamma'\in \Z^d}
\end{align*}
is a bounded operator on $L^2(\R^d)$. Combining Lemma~\ref{lem:directsumofoperatorsbounded} with Lemma~\ref{lem:3} also gives the following corollary.
\begin{corollary}	\label{cor:1}
The operators $\cA_{b,\epsilon}$ are uniformly bounded for $\epsilon\in\,\, ]0,1]$.
\end{corollary}
Next we prove that $H_b$ is the strong limit of $\cA_{b,\epsilon}$ as $\epsilon\to 0$.
\begin{proposition}\label{prop:strongconvergence1}
	The operators $\cA_{b,\epsilon}$ converge strongly to $H_b$ as $\epsilon$ goes to zero.
\end{proposition}
\begin{proof}
First one shows the strong convergence for elements in the set
\begin{align*}
\sH_0^\infty\coloneqq \{(f_\gamma)\in \sH\mid f_\gamma\in C^\infty_0(\Omega)\textup{ and } f_\gamma\neq 0 \textup{ for only finitely many }\gamma\in \Z^d \}
\end{align*}
by using \eqref{eq:prop1} and Lemma~\ref{lem:ben1}. Second one uses that $\sH_0^\infty$ is dense in $\sH$, and that the operators $\cA_{b,\epsilon}$ are uniformly bounded in $\epsilon$ to complete the proof.
\end{proof}

Finally, we are ready to show that $\Op(a_b)$ has a continuous extension on $L^2(\R^d)$. By Proposition~\ref{prop:strongconvergence1} it follows that $U_b^*H_bU_b$ is the strong limit of $\Op(a_{b,\epsilon})$ and since using Lebesgue's dominated convergence theorem in the definition of $\Op(a_{b,\epsilon})$ gives
\begin{align*}
\lim_{\epsilon\to 0} \pair{\Op(a_{b,\epsilon})f}{g}=\pair{\Op(a_b)f}{g},
\end{align*}
for every $f,g\in \SR$ it follows that $U_b^* H_b U_b$ is a continuous extension of $\Op(a_b)$ to $L^2(\R^d)$.
\begin{remark}\label{rem:1<<<<<<<<<<<<<<<<}
	Note that if we had used a general magnetic symbol like in~\eqref{eq:symbol} with $M\geq 0$ then the estimate in~\eqref{eq:agammagammaprimebbound} would be on the form
	\begin{align*}
\babss{\partial_{x}^\alpha\partial_{x'}^{\alpha'}\partial_{\xi}^{\beta}\Big(g(x)g(x')a(x+\gamma,x'+\gamma',\xi)\e^{\I b \fl_{\gamma,\gamma'}(x,x')}\Big)}&\leq C_{\alpha,\alpha',\beta}\jn{\gamma-\gamma'}^{M+\abs{\alpha}+\abs{\alpha'}},
	\end{align*}
	for $x,x'\in \tilde{\Omega}$. Thus the Fourier coefficient obeys
	\begin{align*}
	\abs{\partial_{\xi}^{\beta}\tilde{a}_{b}(\xi)}&\leq C_{\beta}\jn{\gamma-\gamma'}^{4d+M},
	\end{align*}
	instead of~\eqref{eq:tildeagammagammaprimemmprimeNestimate}. The subsequent part of the proof would then follow in exactly the same way with only minor changes e.g.\ replacing $4d$ with $4d+M$.
\end{remark}

\section{Proof of Theorem~\ref{thm:main}(2)}\label{sec:proof2}
In order to prove the second part of Theorem~\ref{thm:main} we introduce the following notation. Define 
\begin{align*}
V_{t}\coloneqq \begin{cases}
\{(\gamma,\gamma')\in \Z^{2d}\mid \abs{\gamma-\gamma'}< \abs{t}^{-1/2}\},& t\neq 0,\\
\Z^d,&t =0.
\end{cases}
\end{align*}
Furthermore, for $s,t\in \R$ define
\begin{align}\label{eq:Hbnotation}
H^{s}_{t,b}\coloneqq \{\e^{\I(b+s) \phi(\gamma,\gamma')}\cA_{\gamma\gamma',b}\}_{(\gamma,\gamma')\in V_t}.
\end{align}
If $s$ or $t$ is $0$ then we omit them in the above notation. Recall that for this part of the proof we assume
\begin{align*}
\overline{a(x,x',\xi)}=a(x',x,\xi),
\end{align*}
for all $x,x',\xi\in \R^d$. This is a sufficient condition for $H_{t,b}^s$ to be self-adjoint for every $s,t\in \R$ and every $b\in [0,b_{\rm max}]$. 

An important result~\cite[Chapter V-§4 theorem 4.10]{Ka} for proving Theorem~\ref{thm:main}(2) is that if $S$ and $T$ are bounded and self-adjoint operators on a Hilbert space then
\begin{align}\label{eq:lax}
\dH(\sigma(S),\sigma(T))\leq \norm{S-T}.
\end{align}

Our strategy to prove~\eqref{eq:main3} is to show that there exists a constant $C$ such that if $b_0\in [0,b_{\rm max}]$ is arbitrary and $\delta b$ satisfies $b_0+\delta b\in [0,b_{\rm max}]$ then
\begin{align}
\dH(\sigma(H_{b_0+\delta b}),\sigma(H_{b_0}^{\delta b}))&\leq C\abs{\delta b}, \label{eq:part21}\\
\dH(\sigma(H_{b_0}^{\delta b}),\sigma(H_{\delta b,b_0}^{\delta b}))&\leq C\abs{\delta b}, \label{eq:part22}\\
\dH(\sigma(H_{\delta b,b_0}^{\delta b}),\sigma(H_{\delta b,b_0}))&\leq C\abs{\delta b}^{1/2}, \label{eq:part24}\\
\dH(\sigma(H_{\delta b,b_0}),\sigma(H_{b_0}))&\leq C\abs{\delta b}. \label{eq:part23}
\end{align}
Since $\abs{\delta b}\in [0,b_{\rm max}]$ the triangle inequality would then imply~\eqref{eq:main3}. Note that the constant $C$ will depend on $b_{\rm max}$. For the rest of this section let $b_0\in [0,b_{\rm max}]$ be arbitrary and let $\delta b$ be sufficiently small.

For the inequality~\eqref{eq:part21} note that
\begin{align*}
H_{b_0+\delta b}-H_{b_0}^{\delta b}=\{\e^{\I(b_0+\delta b) \phi(\gamma,\gamma')}(\cA_{\gamma\gamma',b_0+\delta b}-\cA_{\gamma\gamma',b_0})\}_{\gamma,\gamma'\in \Z^d}.
\end{align*}
Thus it follows from Lemma~\ref{lem:directsumofoperatorsbounded},~\eqref{eq:AggbLipschitz} and~\eqref{eq:lax} that there exists $C$ not depending on $b_0$ or $\delta b$ such that
\begin{align}\label{eq:KS1}
\dH(\sigma(H_{b_0+\delta b}),\sigma(H_{b_0}^{\delta b}))\leq \norm{H_{b_0+\delta b}-H_{b_0}^{\delta b}}\leq C\abs{\delta b}.
\end{align}
The proofs of~\eqref{eq:part22} and~\eqref{eq:part23} are similar hence we only do it for~\eqref{eq:part22}. Clearly,
\begin{align*}
H_{b_0}^{\delta b}-H_{\delta b,b_0}^{\delta b}= \{\e^{\I (b_0+\delta b) \phi(\gamma,\gamma')}\cA_{\gamma\gamma',b_0}\}_{(\gamma,\gamma')\notin V_{\delta b}},
\end{align*}
thus by defining $\tilde{V}_{\delta b}=\Z^d\cap B_{\abs{\delta b}^{-1/2}}(0)$ it follows from Lemma~\ref{lem:directsumofoperatorsbounded},~\eqref{eq:Aggbdecay} and~\eqref{eq:lax} that
\begin{align}\label{eq:b1}
\dH(\sigma(H_{b_0}^{\delta b}),\sigma(H_{\delta b,b_0}^{\delta b}))\leq\norm{H_{b_0}^{\delta b}-H_{\delta b,b_0}^{\delta b}}\leq\sum_{\gamma\notin \tilde{V}_{\delta b}} C\jn{\gamma}^{-2d}.
\end{align}
It is possible to find a constant $C$ such that for all $\gamma\in \Z^d$ we have
\begin{align*}
\jn{\gamma}^{-2d}\leq C\jn{x}^{-2d},
\end{align*}
for all $x\in \gamma+\Omega$. If we dominate the sum in \eqref{eq:b1} by the integral of $C\jn{x}^{-2d}$ and switch to polar coordinates we obtain
\begin{align*}
\dH(\sigma(H_{b_0}^{\delta b}),\sigma(H_{\delta b,b_0}^{\delta b}))\leq \sum_{\gamma\notin \tilde{V}_{\delta b}} C\jn{\gamma}^{-2d}\leq C\int_{\abs{\delta b}^{-1}/2}^\infty \frac{r^{d-1}}{\jn{r}^{2d}}\d r\leq C\abs{\delta b}
\end{align*}
for sufficiently small $\delta b$.
\subsection{Strategy for the proof of \texorpdfstring{\eqref{eq:part24}}{(3.5)}}
The proof of~\eqref{eq:part24} is more involved than the other three estimates since it is not possible in general to bound $\norm{H_{\delta b,b_0}^{\delta b}-H_{\delta b,b_0}}$ by a constant multiple of $\abs{\delta b}$. Our strategy is to prove the following two  results:
\begin{lemma}\label{lem:hausdorffdistancelemma}
	There exists a constant $C>0$ such that if $\dist(z,\sigma(H_{\delta b,b_0}))>C\abs{\delta b}^{1/2}$ then $z\in \rho(H_{\delta b,b_0}^{\delta b})$.
\end{lemma}
\begin{lemma}\label{lem:hausdorffdistancelemma1}
	There exists a constant $C>0$ such that if $\dist(z,\sigma(H_{\delta b,b_0}^{\delta b}))>C\abs{\delta b}^{1/2}$ then $z\in \rho(H_{\delta b,b_0})$.
\end{lemma}
Then \eqref{eq:part24} is a direct consequence of the following general lemma.
\begin{lemma}\label{lem:bs1}
Let $T_1,T_2$ be bounded operators on some Hilbert space and $C >0$ a constant. The following assertions are equivalent:
\begin{enumerate}
	\item If $\dist(z,\sigma(T_j))>C$ then $z\in \rho(T_k)$, for $j,k=1,2$.
	\item $\dH(\sigma(T_1),\sigma(T_2))\leq C$.
\end{enumerate}
\end{lemma}
\begin{proof}
We first show by contradiction that \emph{(1)} implies \emph{(2)}. Assume that $\dH(\sigma(T_1),\sigma(T_2))> C$. Then either there exists some $z$ such that $\dist(z,\sigma(T_2)) > C$ and $z \in \sigma(T_1)$, or there exists some $z$ such that $\dist(z,\sigma(T_1)) > C$ and $z \in \sigma(T_2)$. This  contradicts \emph{(1)}.

To show that \emph{(2)} implies \emph{(1)}, let $z$ be such that $\dist(z,\sigma(T_j))>C$. Then $z$ cannot belong to the spectrum of $T_k$ without contradicting \emph{(2)}.
\end{proof}
In what follows we only prove Lemma~\ref{lem:hausdorffdistancelemma} since the proof of Lemma~\ref{lem:hausdorffdistancelemma1} is similar (cf.\ Remark~\ref{rem:rembs1}). 

The main idea behind the proof of  Lemma~\ref{lem:hausdorffdistancelemma} is showing that for every $z\in \rho(H_{\delta b,b_0})$ there exists some bounded operator $S_z$ such that
\begin{equation}\label{eq:almostident}
(H_{\delta b,b_0}^{\delta b}-z)S_z=\id +\cO\Big(\frac{\abs{\delta b}^{\frac{1}{2}}}{\dist(z,\sigma(H_{\delta b,b_0}))}\Big).
\end{equation} 
Then if the right hand side is invertible, $z$ belongs to the resolvent set of $H_{\delta b,b_0}^{\delta b}$.
\subsection{Proof of Lemma~\ref{lem:hausdorffdistancelemma}}
In order to construct the operator $S_z$ let $g,\tilde{g}\in C_0^\infty(\R^d)$ and $r>0$ satisfy:
\begin{enumerate}
	\item $g(x),\tilde{g}(x)\in [0,1]$ for every $x\in \R^d$. 
	\item $\Supp g\subset B_r(0)$ and $\Supp \tilde{g} \subset B_{r+2}(0)$.
	\item $ \tilde{g}\equiv 1 $ on $B_{r+1}(0)$.
	\item $ \sum_{\gamma\in \Z^d}g^2(x-\gamma) =1$ for every $x\in \R^d$.
\end{enumerate}
Furthermore, for any $n\in\Z^d$ define
\begin{align*}
g_{n,\delta b}(x)\coloneqq g(\abs{\delta b}^{1/2} x-n)\quad \textup{and} \quad \tilde{g}_{n,\delta b}(x)\coloneqq \tilde{g}(\abs{\delta b}^{1/2} x-n)
\end{align*}
and note the following properties:
\begin{enumerate}[label={(\alph*)}]
	\item $ \Supp g_{n,\delta b} \subset B_{r\abs{\delta b}^{-1/2}}(n\abs{\delta b}^{-1/2}) $ and $ \Supp \tilde{g}_{n,\delta b} \subset B_{(r+2)\abs{\delta b}^{-1/2}}(n\abs{\delta b}^{-1/2}) $.
	\item \label{it:meanvalue} $ \abs{g_{n,\delta b}(x)-g_{n,\delta b}(y)}\leq \abs{\delta b}^{1/2} C_g \abs{x-y} $ for every $x,y\in \R^d$. 
	\item $\tilde{g}_{n,\delta b}(x)g_{n,\delta b}(y)=g_{n,\delta b}(y)$ whenever $\abs{x-y}\leq \abs{\delta b}^{-1/2}$.
	\item \label{it:tildegneps} \label{it:gneps} If for each $n\in \Z^d$ we define the set of \emph{$r$-neighbors} to $n$ by 
	\begin{align*}
	N_r(n)\coloneqq\{n'\in \Z^d\mid 0<\abs{n-n'}< 2r\},
	\end{align*}
	then $g_{n,\delta b}g_{n',\delta b}\equiv 0$ if $n'\not \in N_r(n)\cup \{n\}$ and $\tilde{g}_{n,\delta b}\tilde{g}_{n',\delta b}\equiv 0$ if $n'\not \in N_{r+2}(n)\cup \{n\}$.
	\item \label{it:tildegestimat} $\abs{\gamma''-n\abs{\delta b}^{-1/2}}\tilde{g}_{n,\delta b}(\gamma'')\leq (r+2)\abs{\delta b}^{-1/2}\tilde{g}_{n,\delta b}(\gamma'')$ for any $n,\gamma''\in \Z^d$.
\end{enumerate}
For each $n,\gamma\in \Z^d$ define the scalars
\begin{align*}
g_{\gamma,n,\delta b}^{\pm}\coloneqq \e^{\pm \I \delta b \phi(\gamma,n\abs{\delta b}^{-1/2})}g_{n,\delta b}(\gamma)
\end{align*}
and the operator $W_{\delta b}$ on $B(\sH)$ by
\begin{align*}
W_{\delta b}(R)\coloneqq \Big\{\sum_{n\in \Z^d} g_{\gamma,n,\delta b}^{+} R_{\gamma,\gamma'} g_{\gamma',n,\delta b}^{-}\Big\}_{\gamma,\gamma'\in\Z^d},
\end{align*}
for $R\in B(\sH)$.
\begin{lemma}\label{lem:Weps}
	The operator $W_{\delta b}$ is bounded with $\norm{W_{\delta b}}\leq (v_r+1)^{1/2}$, where $v_r\coloneqq \abs{N_r(n)}$ is independent of $n$.
\end{lemma}
\begin{proof}
	Let $f=(f_\gamma)\in \sH$ be arbitrary and for every $n\in \Z^d$ let $\Psi_{n,\delta b}\in \sH$ be given by
	\begin{equation}\label{eq:psi1}
	(\Psi_{n,\delta b})_\gamma:=g_{\gamma,n,\delta b}^{-} f_\gamma.
	\end{equation}
	Then 
	\begin{equation}\label{eq:tempop}
	\sum_{n\in \Z^d}\norm{\Psi_{n,\delta b}}^2_{\sH}=\sum_{n\in \Z^d}\sum_{\gamma\in\Z^d}g_{n,\delta b}^2(\gamma) \norm{f_\gamma}^2_{L^2(\Omega)}=\norm{f}^2_{\sH}.
	\end{equation}
	Let $R\in B(\sH)$ be arbitrary. By the definition of $W_{\delta b}$ we have 
	\begin{equation*}
	[W_{\delta b}(R)f]_{\gamma}=\sum_{n\in \Z^d} g_{\gamma,n,\delta b}^{+}(R\Psi_{n,\delta b})_{\gamma},
	\end{equation*}
	for any $\gamma\in\Z^d$. Thus, if we write the norm of $[W_{\delta b}(R)f]_\gamma$ in $L^2(\Omega)$ as an inner product with the previous expression we obtain the estimate
	\begin{align*}
	\norm{[W_{\delta b}(R)f]_{\gamma}}_{L^2(\Omega)}^2&\leq\sum_{n\in \Z^d}\sum_{n'\in N_r(n)\cup\{n\}}\frac{1}{2}g_{n,\delta b}(\gamma)g_{n',\delta b}(\gamma)(\norm{(R\Psi_{n,\delta b})_{\gamma}}_{L^2(\Omega)}^2+\norm{(R\Psi_{n',\delta b})_{\gamma}}_{L^2(\Omega)}^2)\\
	&\leq \frac{1}{2}\sum_{n\in \Z^d}\bigg((v_r+2)\norm{(R\Psi_{n,\delta b})_{\gamma}}_{L^2(\Omega)}^2+\sum_{n'\in N_r(n)}\norm{(R\Psi_{n',\delta b})_{\gamma}}_{L^2(\Omega)}^2\bigg),
	\end{align*}
	where it suffices to sum $n'$ over the set $N_r(n)\cup\{n\}$ by~\ref{it:gneps}.
	
	For any $n\in \Z^d$ the second sum contains the term $\norm{(R\Psi_{n,\delta b})_{\gamma}}^2_{L^2(\Omega)}$ once for every element in the set $N_r(n)$. Hence we obtain
	\begin{align*}
	\norm{[W_{\delta b}(R)f]_{\gamma}}^2_{L^2(\Omega)}
	&\leq (v_r+1)\sum_{n\in \Z^d}\norm{(R\Psi_{n,\delta b})_{\gamma}}^2_{L^2(\Omega)}.
	\end{align*} 
	By summing over $\gamma\in \Z^d$, and applying the boundedness of $R$ together with~\eqref{eq:tempop} we obtain
	\begin{align*}
	\norm{W_{\delta b}(R)f}^2_{\sH}&\leq (v_r+1)\sum_{n\in \Z^d} \norm{R\Psi_{n,\delta b}}^2_{\sH}\leq (v_r+1) \norm{R}^2 \norm{f}_{\sH}^2,
	\end{align*}
	which completes the proof.
\end{proof}

We will show that the operator $W_{\delta b}((H_{\delta b,b_0}-z)^{-1})$ acts as $S_z$ in \eqref{eq:almostident}. To show this we need the following result.
\begin{lemma}\label{lem:xseq}
	Let $f\in \sH$ and $z\in\rho(H_{\delta b,b_0})$ be arbitrary. For each $\gamma \in \Z^d$ define the scalar
	\begin{align*}
	x_{\gamma}\coloneqq\sum_{n\in \Z^d} \tilde{g}_{n,\delta b}(\gamma)\norm{((H_{\delta b,b_0}-z)^{-1} \Psi_{n,\delta b})_{\gamma}}_{L^2(\Omega)},
	\end{align*}
	where $\Psi_{n,\delta b}$ is given by~\eqref{eq:psi1}. Then $x=(x_\gamma)\in \ell^2(\Z^d)$ with
	\begin{align*}
	\norm{x}_{\ell^2(\Z^d)}&\leq \frac{(v_{r+2}+1)^{1/2}}{\dist(z,\sigma(H_{\delta b,b_0}))} \norm{f}_{\sH}.
	\end{align*}
\end{lemma}
\begin{proof}
	By using similar arguments as in the proof of Lemma~\ref{lem:Weps} we obtain 
	\begin{align*}
	\norm{x}_{\ell^2(\Z^d)}^2&\leq (v_{r+2}+1)\sum_{\gamma\in \Z^d}\sum_{n\in \Z^d} \norm{((H_{\delta b,b_0}-z)^{-1} \Psi_{n,\delta b})_{\gamma}}_{L^2(\Omega)}^2\leq \frac{(v_{r+2}+1)\norm{f}_{\sH}^2}{\dist(z,\sigma(H_{\delta b,b_0}))^2},
	\end{align*}
	where we have used the well-known equality
	\begin{align*}
	\norm{(T-z)^{-1}}= \frac{1}{\dist(z,\sigma(T))},
	\end{align*}
	which holds for $T$ normal and $z\in \rho(T)$.
\end{proof}
We are now ready to verify \eqref{eq:almostident}.
\begin{lemma}\label{lem:pseudoinverse}
	There exists a constant $C$ such that for all $z\in \rho(H_{\delta b,b_0})$ the operator
	\begin{align*}
	T_z=(H_{\delta b,b_0}^{\delta b}-z)W_{\delta b}((H_{\delta b,b_0}-z)^{-1})-\id
	\end{align*}
	is bounded on $\sH$ with
	\begin{align*}
	\norm{T_z}\leq \frac{C\abs{\delta b}^{1/2}}{\dist(z,\sigma(H_{\delta b,b_0}))}.
	\end{align*}
\end{lemma}
\begin{proof}
	To shorten our notation we write
	\begin{align*}
	R_{\delta b,z}\coloneqq(H_{\delta b,b_0}-z)^{-1}.
	\end{align*}
	In order to prove this result we want to obtain the following decomposition 
	\begin{align}\label{eq:kasper}
	(H_{\delta b,b_0}^{\delta b}-z)W_{\delta b}(R_{\delta b,z})&= R_1+R_2+R_3
	\end{align}
	where
	\begin{align*}
	R_1&\coloneqq \Big\{\sum_{\gamma''\in \Z^d}\e^{\I \delta b \phi(\gamma,\gamma'')}[H_{\delta b,b_0}-z]_{\gamma,\gamma''}[W^{(1)}_{\delta b,\gamma}]_{\gamma'',\gamma'}\Big\}_{\gamma,\gamma'\in \Z^d},\\
	R_2&\coloneqq\Big\{\sum_{\gamma''\in \Z^d}[H_{\delta b,b_0}-z]_{\gamma,\gamma''}[W^{(2)}_{\delta b,\gamma}]_{\gamma'',\gamma'}\Big\}_{\gamma,\gamma'\in \Z^d},\\
	R_3&\coloneqq\Big\{\sum_{\gamma''\in \Z^d}[H_{\delta b,b_0}-z]_{\gamma,\gamma''}[W^{(3)}_{\delta b,\gamma}]_{\gamma'',\gamma'}\Big\}_{\gamma,\gamma'\in \Z^d},
	\end{align*}
	for some suitable operators $W_{\delta b,\gamma}^{(1)}, W_{\delta b,\gamma}^{(2)}, W_{\delta b,\gamma}^{(3)}$. To finish the proof we will then show that $R_3=\id$ and that
	\begin{align}\label{eq:Rbound}
	\max\{\norm{R_1},\norm{R_2}\}\leq\frac{C\abs{\delta b}^{1/2}}{\dist(z,\sigma(H_{\delta b,b_0}))},
	\end{align}
	for some constant $C$.
	
	We start by constructing the operators $W_{\delta b,\gamma}^{(1)}, W_{\delta b,\gamma}^{(2)}, W_{\delta b,\gamma}^{(3)}$. Since 
	\begin{align*}
	H_{\delta b,b_0}^{\delta b}-z= \{\e^{\I \delta b \phi(\gamma,\gamma')}[H_{\delta b,b_0}-z]_{\gamma,\gamma'}\}_{(\gamma,\gamma')\in V_{\delta b}},
	\end{align*}
	and $[H_{\delta b,b_0}-z]_{\gamma,\gamma'}=0$ whenever $\abs{\gamma-\gamma'}\geq \abs{\delta b}^{-1/2}$ these operators must be chosen such that for arbitrary $\gamma',\gamma''\in \Z^d$ we have
	\begin{align*}
	[W_{\delta b}(R_{\delta b,z})]_{\gamma'',\gamma'}= [W_{\delta b,\gamma}^{(1)}]_{\gamma'',\gamma'}+\e^{-\I \delta b \phi(\gamma,\gamma'')}([W^{(2)}_{\delta b,\gamma}]_{\gamma'',\gamma'}+[W^{(3)}_{\delta b,\gamma}]_{\gamma'',\gamma'})
	\end{align*}
	whenever $\abs{\gamma-\gamma''}<\abs{\delta b}^{-1/2}$. 
	By~\ref{it:tildegneps} and the identity
	\begin{align*}
	\e^{\I \delta b (\phi(\gamma,\gamma'')+\phi(\gamma'',n\abs{\delta b}^{-1/2}))}=e ^{\I \delta b \phi(\gamma,n\abs{\delta b}^{-1/2})}(1+\e^{\I \delta b \fl(\gamma,\gamma'',n\abs{\delta b}^{-1/2})}-1),
	\end{align*}
	which hold for all $\gamma,\gamma''\in \Z^d$, it is possible to verify that defining
	\begin{align*}
	W_{\delta b,\gamma}^{(1)}&\coloneqq\!\Big\{\!\sum_{n\in \Z^d} \e^{\I \delta b \phi(\gamma'',n\abs{\delta b}^{-1/2})}(g_{n,\delta b}(\gamma'')- g_{n,\delta b}(\gamma))\tilde{g}_{n,\delta b}(\gamma'')(R_{\delta b,z})_{\gamma'',\gamma'}g_{\gamma',n,\delta b}^{-}\!\Big\}_{\!\gamma'',\gamma'\in \Z^d},\\
	W_{\delta b,\gamma}^{(2)}&\coloneqq \! \Big\{\!\sum_{n\in \Z^d}\e^{\I \delta b \phi(\gamma,n\abs{\delta b}^{-1/2})} (\e^{\I \delta b \fl(\gamma,\gamma'',n\abs{\delta b}^{-1/2})}-1)g_{n,\delta b}(\gamma)\tilde{g}_{n,\delta b}(\gamma'')(R_{\delta b,z})_{\gamma'',\gamma'}g_{\gamma',n,\delta b}^{-}\!\Big\}_{\!\gamma'',\gamma'\in \Z^d},\\
	W_{\delta b,\gamma}^{(3)}&\coloneqq\! \Big\{\!\sum_{n\in \Z^d} \e^{\I \delta b \phi(\gamma,n\abs{\delta b}^{-1/2})}g_{n,\delta b}(\gamma)\tilde{g}_{n,\delta b}(\gamma'')(R_{\delta b,z})_{\gamma'',\gamma'}g_{\gamma',n,\delta b}^{-}\!\Big\}_{\!\gamma'',\gamma'\in \Z^d},
	\end{align*}
	gives the desired decomposition of~\eqref{eq:kasper}. 
	
	By using the definition of $W_{\delta b,\gamma}^{(3)}$ it follows that $R_3=\id$. To achieve estimate~\eqref{eq:Rbound} let $f=(f_\gamma)\in \sH$ be arbitrary. Our strategy is to bound the quantity $\norm{(R_jf)_\gamma}_{L^2(\Omega)}$, $j=1,2$, by a product of an operator in $B(\ell^2(\Z^d))$ and a vector in $\ell^2(\Z^d)$.
	
	Let $S\colon \ell^2(\Z^d)\to \ell^2(\Z^d)$ be the integral operator with kernel
	\begin{align*}
	S(\gamma,\gamma')=\abs{\gamma-\gamma'}\norm{[H_{\delta b,b_0}-z]_{\gamma,\gamma'}}_{L^2(\Omega)},
	\end{align*}
	and let $x=(x_\gamma)\in \ell^2(\Z^d)$ be given as in Lemma~\ref{lem:xseq}, i.e.\
	\begin{align*}
	x_{\gamma}\coloneqq\sum_{n\in \Z^d} \tilde{g}_{n,\delta b}(\gamma)\norm{((H_{\delta b,b_0}-z)^{-1} \Psi_{n,\delta b})_{\gamma}}_{L^2(\Omega)}.
	\end{align*}
	By~\ref{it:meanvalue} and the triangle inequality we get
	\begin{align*}
	\norm{(R_1f)_\gamma}_{L^2(\Omega)}&\leq\abs{\delta b}^{1/2}(Sx)_{\gamma},
	\end{align*}
	and from~\ref{it:tildegestimat},~\eqref{eq:fluxbound} and~\eqref{eq:complexexponentialinequality} we obtain
	\begin{align*}
	\norm{(R_2f)_\gamma}_{L^2(\Omega)}&\leq C_{r}\abs{\delta b}^{1/2}(Sx)_{\gamma},
	\end{align*}
	for some appropriate constant $C_{r}$. From~\eqref{eq:Aggbdecay} and a Schur-Holmgren type result for $\ell^2(\Z^d)$ it follows that $S$ is bounded. By Lemma~\ref{lem:xseq} we thus obtain the bound~\eqref{eq:Rbound} for both $R_1$ and $R_2$.
\end{proof}

\begin{proof}[Proof of Lemma~\ref{lem:hausdorffdistancelemma}]
	Since $H_{\delta b,b_0}^{\delta b}$ is self-adjoint it suffices to consider only real values of $z$. Suppose that $x\in \R$ with $\dist(x,\sigma(H_{\delta b,b_0}))>2C\abs{\delta b}^{1/2}$ and choose $\delta_0>0$ such that $z\in \rho(H_{\delta b,b_0})$ whenever $\abs{z-x}<\delta_0$. For any $\delta\in\R$ with $0<\abs{\delta}<\delta_0$ we define $z_\delta=x+\I \delta$. By Lemma~\ref{lem:Weps} and Lemma~\ref{lem:pseudoinverse} we have the estimates
	\begin{align*}
	\norm{W_{\delta b}((H_{\delta b,b_0}-z_\delta)^{-1})}\leq \frac{(v_r+1)^{1/2}}{\dist(z_\delta,\sigma(H_{\delta b,b_0}))}\leq \frac{(v_r+1)^{1/2}}{\dist(x,\sigma(H_{\delta b,b_0}))},
	\end{align*}
	and
	\begin{align*}
	\norm{T_{z_\delta}}\leq \frac{C\abs{\delta b}^{1/2}}{\dist(z_\delta,\sigma(H_{\delta b,b_0}))}<\frac{1}{2},
	\end{align*}
	for all $0<\abs{\delta}<\delta_0$. Using these estimates together with Lemma~\ref{lem:pseudoinverse} gives 
	\begin{align*}
	(H_{\delta b,b_0}^{\delta b}-z_\delta)^{-1} =W_{\delta b}((H_{\delta b,b_0}-z_\delta)^{-1})(\id+T_{z_\delta})^{-1}
	\end{align*}
	and that $(H_{\delta b,b_0}^{\delta b}-z_\delta)^{-1}$ is bounded uniformly for such $\delta$. Factorizing
	\begin{align*}
	H_{\delta b,b_0}^{\delta b}-x=(\id+\I \delta(H_{\delta b,b_0}^{\delta b}-z_\delta)^{-1})(H_{\delta b,b_0}^{\delta b}-z_\delta),
	\end{align*}
	and choosing $\delta$ sufficiently small concludes the proof.
\end{proof}

\begin{remark}\label{rem:rembs1}
If we define the operator $\tilde{W}_{\delta b}$ on $B(\sH)$ by  
\begin{align*}
\tilde{W}_{\delta b}(R)\coloneqq\Big \{\sum_{n\in \Z^d} g_{\gamma,n,\delta b}^{-}R_{\gamma,\gamma'} g_{\gamma',n,\delta b}^{+}\Big\}_{\gamma,\gamma'\in \Z^d},
\end{align*}
and interchange the roles of $H_{\delta b,b_0}^{\delta b}$ and $H_{\delta b,b_0}$ it is possible to repeat the proofs of Lemma~\ref{lem:Weps}, Lemma~\ref{lem:xseq}, Lemma~\ref{lem:pseudoinverse} and Lemma~\ref{lem:hausdorffdistancelemma} to obtain the result in Lemma~\ref{lem:hausdorffdistancelemma1}.
\end{remark}

\section{Proof of Theorem~\ref{thm:main}(3)}\label{sec:proof3}
In this part of the proof we adopt the notation in~\eqref{eq:Hbnotation}. Recall that we now assume that $B$ is a constant magnetic field. Thus $\phi$ is bilinear and 
\begin{align}\label{eq:phiidentity} 
\phi(x,y)+\phi(y,z)=\phi(x,z)+\phi(x-y,y-z),
\end{align}
for all $x,y,z\in \R^d$. 
\subsection{Regularity of extremal spectral values}
Let $b_0,b_0+\delta b\in [0,b_{\rm max}]$ for an arbitrary $b_0$ and sufficiently small $\delta b$. We only consider the case when $E_b$ is the maximum of the spectrum, the case when $E_b$ is the minimum is similar.  By~\eqref{eq:part21} there exists a constant $C$ such that
\begin{align*}
\abs{E_{b_0+\delta b} - \sup \sigma(H^{\delta b}_{b_0})} \leq \dH\big(\sigma(H_{b_0+\delta b}), \sigma(H^{\delta b}_{b_0})\big) &\leq C \abs{\delta b}\numberthis \label{eq:c1absepsestimate1}
\end{align*}
and by the triangle inequality and \eqref{eq:c1absepsestimate1} we get
\begin{align*}
\abs{E_{b_0+\delta b}-E_{b_0}} &\leq \abs{E_{b_0+\delta b}-\sup \sigma (H^{\delta b}_{b_0})} + \abs{ \sup \sigma ( H_{b_0}^{\delta b}) - E_{b_0}} \\
&\leq C\abs{\delta b} + \abs{\sup \sigma(H^{\delta b}_{b_0})-\sup \sigma(H_{b_0})}.
\end{align*}
Thus, it only remains to prove the following lemma.
\begin{lemma}\label{lem:lem}
	There exists some constant $C$ such that
	\begin{align}
	\sup \sigma(H^{\delta b}_{b_0}) &\leq \sup \sigma(H_{b_0}) + C\abs{\delta b},\label{eq:123321} \\
	\sup \sigma(H_{b_0}) &\leq \sup \sigma(H^{\delta b}_{b_0}) + C\abs{\delta b},\label{eq:321123}
	\end{align}
	hence
	\begin{align*}
	\abs{\sup \sigma(H^{\delta b}_{b_0})-\sup \sigma(H_{b_0})}\leq C\abs{\delta b}.
	\end{align*}
\end{lemma}
Before we prove this proposition we consider the fundamental solution to the heat equation, as it is an essential part of the proof. The fundamental solution is given by 
\begin{align}\label{eq:fundamentalsolutionheatequation}
G(y,y',t)=\frac{1}{(4\pi t)^{d/2}}\e^{-\abs{y-y'}^2/4t}
\end{align}
which is symmetric in the spatial coordinates and by semi-group theory satisfies 
\begin{align}\label{eq:fundamentalsolutionheateqidentity}
G(y,y'',2t)=\int_{\R^\dim}G(y,y',t)G(y',y'',t)\d y'.
\end{align}
By letting $y=y''$ we get that
\begin{align}\label{eq:fundamentalheatequationyequaltydobbeltdot}
\int_{\R^\dim} \abs{G(y,y',t)}^2\d y'=G(y,y,2t) =\frac{1}{(8\pi t)^{\dim/2}}.
\end{align}
To simplify our notation we define the linear functional $\Lambda_{\gamma,\gamma',t}$ by 
\begin{align*}
\Lambda_{\gamma,\gamma',t}f\coloneqq\int_{\R^d}f(y') G(\gamma,y',t)G(y',\gamma',t)\d y'.
\end{align*}
By \eqref{eq:phiidentity}, \eqref{eq:fundamentalsolutionheatequation} and \eqref{eq:fundamentalsolutionheateqidentity} we get
\begin{align*}
\Lambda_{\gamma,\gamma',t}(\e^{\I\delta b\phi(\gamma,\cdot)}\e^{\I\delta b\phi(\cdot,\gamma')})& =\Lambda_{\gamma,\gamma',t} \Bigg(\e^{\I\delta b\phi(\gamma,\gamma')} \Big(\e^{\I\delta b\phi(\gamma-\cdot,\cdot-\gamma')}-1+1\Big)\Bigg)\\
&=\e^{\I\delta b\phi(\gamma,\gamma')}G(\gamma,\gamma',2t) + \e^{\I\delta b\phi(\gamma,\gamma')} \Lambda_{\gamma,\gamma',t}\Big(\e^{\I\delta b\phi(\gamma-\cdot,\cdot-\gamma')}-1\Big)\\
&=\e^{\I\delta b\phi(\gamma,\gamma')} \bigg(\frac{1}{(8\pi t)^{d/2}}+\frac{1}{(8\pi t)^{d/2}}\Big(\e^{-\abs{\gamma-\gamma'}^2/8t}-1\Big)\bigg) \\
&\phantom{{}={}}+ \e^{\I\delta b\phi(\gamma,\gamma')} \Lambda_{\gamma,\gamma',t}\Big(\e^{\I\delta b\phi(\gamma-\cdot,\cdot-\gamma')}-1\Big).
\end{align*}
Rearranging the above equation gives for any $\delta b \in \R$ and $\gamma,\gamma'\in \Z^d$
\begin{align*}\label{eq:identityeksponential}
\e^{\I\delta b\phi(\gamma,\gamma')}&= (8\pi t)^{d/2}\Lambda_{\gamma,\gamma',t}(\e^{\I\delta b \phi(\gamma,\cdot)}\e^{\I\delta b\phi(\cdot,\gamma')})- \e^{\I\delta b\phi(\gamma,\gamma')}\bigg[ \Big(\e^{-\abs{\gamma-\gamma'}^2/8t}-1\Big)\\
&\phantom{{}={}}+(8\pi t)^{d/2} \Lambda_{\gamma,\gamma',t}\Big(\e^{\I\delta b\phi(\gamma-\cdot,\cdot-\gamma')}-1\Big) \bigg]\\
&= \mathrm{I}-\e^{\I \delta b \phi(\gamma,\gamma')}[\mathrm{II}+\mathrm{III}],\numberthis
\end{align*}
where
\begin{align*}
\mathrm{I}&:=(8\pi t)^{d/2}\Lambda_{\gamma,\gamma',t}(\e^{\I\delta b \phi(\gamma,\cdot)}\e^{\I\delta b\phi(\cdot,\gamma')}),\\
\mathrm{II}&:=\e^{-\abs{\gamma-\gamma'}^2/8t}-1,\\
\mathrm{III}&:=(8\pi t)^{d/2} \Lambda_{\gamma,\gamma',t}\Big(\e^{\I\delta b\phi(\gamma-\cdot,\cdot-\gamma')}-1\Big).
\end{align*}
\begin{proof}[Proof of Lemma~\ref{lem:lem}]
	Recall that for a self-adjoint operator $T$ on a separable Hilbert space we have
	\begin{align}\label{eq:spectraltheoremequality}
	\sup_{\norm{x}=1} \ip{Tx}{x}=\sup \sigma(T).
	\end{align}
	To show the first inequality, let $f \in \sH$ with $\norm{f}_\sH=1$. By using \eqref{eq:identityeksponential} we get
	\begin{align*}\label{eq:etplusto}
	\ip{H^{\delta b}_{b_0}f}{f}_\sH &= \sum_{\gamma,\gamma' \in \Z^\dim} \e^{\I\delta b\phi(\gamma,\gamma')}\e^{\I b_0\phi(\gamma,\gamma')} \ip{\cA_{\gamma\gamma',b_0}f_{\gamma'}}{f_\gamma}_{L^2(\Omega)}\\
	&=\sum_{\gamma,\gamma'\in \Z^d}(\mathrm{I}-\e^{\I \delta b \phi(\gamma,\gamma')}[\mathrm{II}+\mathrm{III}])\e^{\I b_0\phi(\gamma,\gamma')} \ip{\cA_{\gamma\gamma',b_0}f_{\gamma'}}{f_\gamma}_{L^2(\Omega)}.\numberthis
	\end{align*}
	We first consider the series involving $\mathrm{I}$. Since $G(y',\gamma,t)=G(\gamma,y',t)$ we can define $\Phi_{\delta b,y',t}\in \sH$ by
	\begin{align*}
	(\Phi_{\delta b,y',t})_\gamma\coloneqq \e^{\I\delta b \phi(y',\gamma)}G(y',\gamma,t)f_\gamma,
	\end{align*}
	to get
	\begin{align*}
	\sum_{\gamma,\gamma'\in\Z^d}\mathrm{I}\e^{\I b_0\phi(\gamma,\gamma')} \ip{\cA_{\gamma\gamma',b_0}f_{\gamma'}}{f_\gamma}_{L^2(\Omega)}&=(8\pi t)^{\dim/2}\int_{\R^\dim} \sum_{\gamma \in \Z^\dim} \ip{(H_{b_0} \Phi_{\delta b,y',t})_\gamma}{(\Phi_{\delta b,y',t})_\gamma}_{L^2(\Omega)} \d y' \\
	&=(8\pi t)^{\dim/2}\int_{\R^\dim} \ip{H_{b_0} \Phi_{\delta b,y',t}}{\Phi_{\delta b,y',t}}_{\sH} \d y' \\
	&\leq \sup \sigma(H_{b_0}) (8\pi t)^{\dim/2} \int_{\R^\dim} \sum_{\gamma \in \Z^\dim} \abs{ G(y',\gamma,t)}^2 \norm{f_\gamma}^2_{L^2(\Omega)} \d y' \\
	&= \sup \sigma (H_{b_0}),
	\end{align*}
	where we in the inequality have used \eqref{eq:spectraltheoremequality} and in the last equality \eqref{eq:fundamentalheatequationyequaltydobbeltdot}.
	
	Note that by this and since the left hand side of \eqref{eq:etplusto} is real it follows that the series involving $\mathrm{II}$ and $\mathrm{III}$ must be real.
	
	Next we note that
	\begin{align*}
	\mathrm{II}\leq \babss{\e^{-\abs{\gamma-\gamma'}^2/8t}-1} \leq \frac{\abs{\gamma-\gamma'}^2}{8t}.
	\end{align*}
	
	We now consider $\mathrm{III}$. The antisymmetry of the matrix $B$ in the magnetic field $\phi$ and the coordinate change $x=y'-(\gamma+\gamma')/2$ implies that 
	\begin{align*}
	\Lambda_{\gamma,\gamma't}(\phi(\gamma-\cdot,\cdot-\gamma'))&= \int_{\R^\dim} \phi(\gamma-y',y'-\gamma')G(\gamma,y',t)G(y',\gamma',t) \d y' \\
	&=\frac{1}{(4\pi t)^{d}}\int_{\R^\dim} \phi\Big(\frac{\gamma-\gamma'}{2}-x,x+\frac{\gamma-\gamma'}{2}\Big)\e^{-\frac{1}{4t}(\abs{\frac{\gamma-\gamma'}{2}-x}^2+\abs{x+\frac{\gamma-\gamma'}{2}}^2)} \d x\\
	&=0,
	\end{align*}
	where the last equality comes from the fact that $\phi$ is antisymmetric and the exponential factor is symmetric in $x$ and $\frac{\gamma-\gamma'}{2}$.
	Using this together with the inequality
	\begin{align*}
	\abs{\e^{\I\delta b x} - 1 - \I \delta b x} \leq  \abs{\delta b x}^2,
	\end{align*}
	which holds for $x\in \R$, gives that
	\begin{align*}
	\Lambda_{\gamma,\gamma',t}(\e^{\I\delta b \phi(\gamma-\cdot,\cdot-\gamma')}-1)&\leq  \babs{\Lambda_{\gamma,\gamma',t}(\e^{\I\delta b \phi(\gamma-\cdot,\cdot-\gamma')}-1)- \I \delta b \Lambda_{\gamma,\gamma',t}(\phi(\gamma-\cdot,\cdot-\gamma'))}\\ 
	& \leq  (\delta b)^2 \Lambda_{\gamma,\gamma',t}(\abs{\phi(\gamma-\cdot,\cdot-\gamma')}^2).
	\end{align*}
	Since $\phi(\gamma-y',y'-\gamma)=\phi(\gamma-y'+\gamma'-\gamma',y'-\gamma')=\phi(\gamma-\gamma',y'-\gamma')$ and
	\begin{align*}
	\abs{\phi(\gamma-\gamma',y'-\gamma')}^2\leq C \abs{\gamma-\gamma'}^2 \abs{y'-\gamma'}^2
	\end{align*}
	it follows that 
	\begin{align*}
	(\delta b)^2\Lambda_{\gamma,\gamma',t}(\abs{\phi(\gamma-\cdot,\cdot-\gamma')}^2) &\leq C  (\delta b)^2 \abs{\gamma-\gamma'}^2 \Lambda_{\gamma,\gamma',t}(\abs{y'-\gamma'}^2).
	\end{align*}
	Using that $\abs{y'-\gamma'}^2 \leq \abs{\gamma-y'}^2+\abs{y'-\gamma'}^2$ and changing to polar coordinates implies 
	\begin{align*}
	C  (\delta b)^2 \abs{\gamma-\gamma'}^2 \int_{\R^\dim}G(\gamma,y',t)G(y',\gamma',t)\abs{y'-\gamma'}^2 \d y' \leq C (\delta b)^2 \abs{\gamma-\gamma'}^2 \frac{t}{t^{d/2}}.
	\end{align*}
	Thus we have shown that
	\begin{align*}
	\mathrm{III}=(8\pi t)^{d/2}\Lambda_{\gamma,\gamma',t}(\e^{\I\delta b \phi(\gamma-\cdot,\cdot-\gamma')}-1)\leq C t^{d/2} (\delta b)^2 \abs{\gamma-\gamma'}^2 \frac{t}{t^{d/2}} = C (\delta b)^2\abs{\gamma-\gamma'}^2 t.
	\end{align*}
	Next we define the integral operator $\tilde{S} \colon \ell^2(\Z^\dim) \to \ell^2(\Z^\dim)$ by
	\begin{align*}
	(\tilde{S}x)_\gamma = \sum_{\gamma' \in Z^\dim} \abs{\gamma-\gamma'}^2 \norm{\cA_{\gamma\gamma',b_0}}x_{\gamma'},
	\end{align*}
	which by a Schur-Holmgren type result is a bounded operator.
	
	By inserting the previous estimates for $\mathrm{I},\mathrm{II}$ and $\mathrm{III}$ in \eqref{eq:etplusto} and using that $\tilde{S}\in B(\ell^2(\Z^d))$ we obtain
	\begin{align*}
	\ip{H_{b_0}^{\delta b} f}{f}_{\sH}&\leq \sup \sigma(H_{b_0}) + C\bigg[ \frac{1}{t}+  (\delta b)^2 t\bigg]
	\end{align*}
	Choosing $t= 1/\abs{\delta b}$ finishes the proof of~\eqref{eq:123321}.
	
	To show the second inequality \eqref{eq:321123}, note that by complex conjugation of \eqref{eq:identityeksponential} we obtain
	\begin{align*}
	\ip{H_{b_0} f}{f}_{\sH} &= \sum_{\gamma,\gamma \in \Z^\dim} \e^{\I b_0\phi(\gamma,\gamma')} \ip{\cA_{\gamma\gamma',b_0} f_{\gamma'}}{f_\gamma}_{L^2(\Omega)} \\
	&=\sum_{\gamma,\gamma \in \Z^\dim} \e^{-\I\delta b\phi(\gamma,\gamma')} \e^{\I (b_0+\delta b) \phi(\gamma,\gamma')} \ip{ \cA_{\gamma\gamma',b_0} f_{\gamma'}}{f_\gamma}_{L^2(\Omega)} \\
	&=\sum_{\gamma,\gamma'\in \Z^d}(\overline{\mathrm{I}}-\e^{-\I \delta b \phi(\gamma,\gamma')}[\overline{\mathrm{II}}+\overline{\mathrm{III}}])\e^{\I(b_0+\delta b)\phi(\gamma,\gamma')} \ip{\cA_{\gamma\gamma',b_0}f_{\gamma'}}{f_\gamma}_{L^2(\Omega)},
	\end{align*}
	thus the proof of \eqref{eq:321123} is analogue to the proof of \eqref{eq:123321}.
\end{proof}

\subsection{Regularity of gap edges}
Assume that the spectrum of $H_b$ has a gap i.e.\ $\sigma(H_b)=\sigma_1 \cup \sigma_2$ where $\sup\sigma_1 <\inf\sigma_2$, which does not close when $b$ varies in some interval $[b_1,b_2] \subset [0,b_{\mathrm{max}}]$. We will show that $e_b=\inf\sigma_2$ is Lipschitz continuous in $[b_1,b_2]$. The proof for $\sup \sigma_1$ is similar. 

Without loss of generality, up to a translation in energy, we can assume that $\sigma(H_b)\subset\,\, ]-\infty,0[$ for all $b\in [b_1,b_2]$. Let us fix some $b_0\in\,\, ]b_1,b_2[$ and consider small variations $\delta b$ such that $b_0+\delta b\in [b_1,b_2]$. By the fact that the gap does not close, and if $|\delta b|$ is small enough, we are able to choose a contour $\sC$ around $\sigma_2$ (with $\sigma_1$ exterior to $\sC$) which is independent of $\delta b$ such that the distance between $\sC$ and the spectrum of $H_{b_0+\delta b}$ remains positive, uniformly on $\delta b$. We define the operator
\begin{align*}
T_b\coloneqq\frac{\I}{2\pi}\int_\sC z (H_b - z)^{-1} \d z,
\end{align*}
whose spectrum equals $\sigma_2\cup\{0\}$ and hence $\inf\sigma(T_b)=e_b$. Therefore it is enough to show that the infimum of the spectrum of $T_b$ is Lipschitz continuous in $b$. We will do this in three steps. In what follows,  $C$ denotes  a generic positive constant. 
\subsubsection{Step 1.} 
Consider the operator $H_{b_0}^{\delta b}$ which is defined as in \eqref{eq:mainthm1} but with $\cA_{\gamma\gamma',b_0}$ instead of $\cA_{\gamma\gamma',b_0+\delta b}$, all other phases being left unchanged. From \eqref{eq:KS1} we have  
$$\| H_{b_0+\delta b}-H_{b_0}^{\delta b}\|\leq C |\delta b|.$$ 
Standard perturbation theory arguments imply that if $|\delta b|$ is small enough then $\sC$ is at a positive distance from the spectrum of $H_{b_0}^{\delta b}$ and moreover
\begin{align*}
\Bnorm{T_{b_0+\delta b} - \frac{\I}{2\pi} \int_{\sC}z(H_{b_0}^{\delta b} - z)^{-1} \d z } 
&\leq C \abs{\delta b}.
\end{align*}
Due to \eqref{eq:lax}, it follows that the difference between $e_{b_0+\delta b}$ and the infimum of the spectrum of $\frac{\I}{2\pi} \int_{\sC}z(H_{b_0}^{\delta b} - z)^{-1} \d z$ must be of order $\delta b$.  

\subsubsection{Step 2.} Let $\widetilde{T}_{b_0}^{\delta b}$ be defined as
\begin{align}\label{hc101}
[\widetilde{T}_{b_0}^{\delta b}]_{\gamma,\gamma'}\coloneqq \e^{\I\delta b \phi(\gamma,\gamma')} [T_{b_0}]_{\gamma,\gamma'}.
\end{align}
In what follows we will prove the estimate
\begin{align}\label{hc102}
\Bnorm{\widetilde{T}_{b_0}^{\delta b}-\frac{\I}{2\pi} \int_{\sC}z(H_{b_0}^{\delta b} - z)^{-1} \d z }\leq C |\delta b|,
\end{align}
which when combined with Step 1 and \eqref{eq:lax} gives
\begin{align}\label{hc105}
|e_{b_0+\delta b}-\inf\sigma (\widetilde{T}_{b_0}^{\delta b})|\leq C |\delta b|.
\end{align}

The rest of Step 2 is dedicated to the proof of \eqref{hc102}. 
We start with a technical result.
\begin{lemma}\label{lemma-horia}
	Let $z\in \sC$ and let $b=b_0+\delta b$ as above. Seen as an operator in $\sH=\ell^2(\Z^d;L^2(\Omega))$, the resolvent $(H_b-z)^{-1}$ is also written
	\begin{equation*}
	(H_b-z)^{-1} = \{ [ (H_b - z)^{-1} ]_{\gamma,\gamma'} \}_{\gamma,\gamma' \in \Z^d} \quad\text{with matrix elements}\quad [ (H_b - z)^{-1} ]_{\gamma,\gamma'}\in B(L^2(\Omega)).
	\end{equation*}
	For every $N \in \N$ there exists a constant $C_N$ independent of $b$ and $z$ such that
	\begin{equation*}
	\norm{[ (H_b-z)^{-1}]_{\gamma,\gamma'}} \leq C_N\jn{\gamma-\gamma'}^{-N}.
	\end{equation*}
\end{lemma}
\begin{proof}
Let $k\in \{1,2,\dots,d\}$ and consider the family of unitary operators $V_k(t)\in B(\sH)$ given by
	\begin{equation*}
	(V_k(t)f)_\gamma :=\e^{\I\gamma_k t} f_\gamma.
	\end{equation*}
	The operator $Y_{k,b}(t):= V_k(t) H_b V_k(t)^*$ is isospectral with $H_b$ and 
	\begin{align}\label{hc100}
	V_k(t) (H_b-z)^{-1} V_k(t)^*=\big (Y_{k,b}(t)-z\big )^{-1}.
	\end{align}
	Using \eqref{eq:mainthm1} and \eqref{eq:Aggbdecay}, together with the identity 
$$[Y_{k,b}(t)]_{\gamma,\gamma'}= \e^{\I(\gamma_k-\gamma_k') t}[H_b]_{\gamma,\gamma'},$$	
	 it follows that the map $\R\ni t\mapsto Y_{k,b}(t)$ is infinitely many times differentiable in the norm topology. In particular, 
	 $$ [Y_{k,b}^{(j)}(0)]_{\gamma,\gamma'}=\I^j(\gamma_k-\gamma_k')^j [H_b]_{\gamma,\gamma'},\quad j\geq 1.$$
	 By standard arguments one now shows that the map 
	 $\R\ni t\mapsto \big (Y_{k,b}(t)-z\big )^{-1}$
	 is also differentiable and 
	 
	 $$\frac{\di}{\di t}(Y_{k,b}(t)-z\big )^{-1}=-(Y_{k,b}(t)-z\big )^{-1}Y_{k,b}^{'}(t)(Y_{k,b}(t)-z\big )^{-1}.$$
	 By induction one proves that the resolvent of $Y_{k,b}(t)$ is infinitely many times differentiable. Given $N$, one can express $\frac{\di^N}{\di t^N}(Y_{k,b}(t)-z\big )^{-1}|_{t=0}$ in terms only depending on $(H_b-z)^{-1}$ and $Y_{k,b}^{(j)}(0)$ with $1\leq j\leq N$. Now going back to \eqref{hc100} we see that by fixing a pair $\gamma,\gamma'$ and after differentiating $N$ times at $t=0$ we have:
	 $$ \I^N(\gamma_k-\gamma_k')^N [(H_b-z)^{-1}]_{\gamma,\gamma'}=\left [ \frac{\di^N}{\di t^N}(Y_{k,b}(t)-z\big )^{-1}|_{t=0}\right ]_{\gamma,\gamma'}.$$
	 Since the right hand side is uniformly bounded in $k$, $\gamma$ and $\gamma'$, the proof is completed by noticing that $\langle \gamma-\gamma'\rangle $ grows like $ \max_{k}\abs{\gamma_k-\gamma_k'}$.
\end{proof}

Define $S_{\delta b}(z)$ to be given by: 
$$ [S_{\delta b}(z)]_{\gamma,\gamma'}:=\e^{\I\delta b \varphi (\gamma,\gamma')} [(H_{b_0}-z)^{-1}]_{\gamma,\gamma'}.$$

Since both $H_{b_0}-z$ and $(H_{b_0}-z)^{-1}$ are strongly localized near the diagonal we get
\begin{align*}
[(H_{b_0}^{\delta b} - z)S_{\delta b}(z)]_{\gamma,\gamma''} &= \sum_{\gamma' \in \Z^d} \e^{\I\delta b\phi(\gamma,\gamma'')}\e^{\I\delta b\phi(\gamma-\gamma',\gamma'-\gamma'')} [ H_{b_0}-z ]_{\gamma,\gamma'} [ (H_{b_0}-z)^{-1}]_{\gamma',\gamma''} \\
&=[\id + \cO(\delta b)]_{\gamma,\gamma''}
\end{align*}
and for sufficiently small $|\delta b|$ we obtain
\begin{align*}
(H_{b_0}^{\delta b} - z)^{-1} = S_{\delta b}(z) (\id + \cO(\delta b))^{-1} = S_{\delta b}(z) + \cO(\delta b)
\end{align*}
uniformly in $z\in \sC$. 
By using this identity it follows that
\begin{align*}
\frac{\I}{2\pi} \int_\sC z(H_{b_0}^{\delta b}-z)^{-1} \d z &= \frac{\I}{2\pi}\int_\sC z \; S_{\delta b}(z)  \d z + \cO(\delta b) =\widetilde{T}_{b_0}^{\delta b} + \cO(\delta b) 
\end{align*}
which finishes the proof of \eqref{hc102}. 

\subsubsection{Step 3.} Due to \eqref{hc105} it is enough to prove that 
$$|\inf \sigma(\widetilde{T}_{b_0}^{\delta b})-e_{b_0}|\leq C|\delta b|.$$
We observe that when $\delta b=0$ we have $\widetilde{T}_{b_0}^{0}=T_{b_0}$, hence the above inequality is the same as
$$|\inf \sigma(\widetilde{T}_{b_0}^{\delta b})-\inf \sigma(\widetilde{T}_{b_0}^{0})|\leq C|\delta b|.$$

We also observe that the family $\widetilde{T}_{b_0}^{\delta b}$ defined in \eqref{hc101} is of the same type as the one we introduced in \eqref{eq:mainthm1}, where $\e^{\I b \varphi(\gamma,\gamma')}$ is replaced with $\e^{\I\delta b \varphi(\gamma,\gamma')}$ and $\cA_{\gamma\gamma',b}$ is replaced with $[T_{b_0}]_{\gamma,\gamma'}$. These operators are strongly localized in $\langle \gamma -\gamma'\rangle$ due to Lemma \ref{lemma-horia}. Thus we may apply the result about the Lipschitz continuity in $b$ of the "global" infimum of the spectrum which we have already studied in the first part of Theorem \ref{thm:main}(3), hence concluding the proof.


\begin{thebibliography}{11}
	
	\bibitem{AMP}     N. Athmouni, M. M\u{a}ntoiu and R. Purice. {\it On the
		continuity of spectra for families of magnetic pseudodifferential operators}.
	J.\ Math.\ Phys.\ {\bf 51}, Article ID 083517 (2010). 
	
	\bibitem{Beals}   R. Beals. {\it Characterization of pseudodifferential operators and applications}. 
Duke Math.\ J.\ {\bf 44}(1), 45--57 (1977).

	\bibitem{BB}      S. Beckus and J. Bellissard. {\it Continuity of the spectrum of a field of self-adjoint operators}. Ann.\ H.\ Poincar{\'e} {\bf 17}(12), 3425--3442 (2016).
	
	\bibitem{B}       J. Bellissard. {\it $C^*$-algebras in solid state physics: 2D electrons in uniform magnetic field}. In D. Evans and M. Takesaki (Eds.), Operator Algebras and Applications {\bf 2}, 49–-76, London Math.\ Soc.\ Lecture Note Ser., 136, Cambridge Univ.\ Press, Cambridge, 1989.
	
	\bibitem{B2}      J. Bellissard. {\it Lipshitz continuity of gap boundaries for Hofstadter-like spectra}. Comm.\ Math.\ Phys.\ {\bf 160}(3), 599--613 (1994).
	
	\bibitem{Bo}      J.M. Bony. {\it Caract\'{e}risation des op\'{e}rateurs pseudo-diff\'{e}rentiels}. Ecole Polytechnique, S\'{e}minaire E.D.P. (1996--1997), Expos\'{e} no. XXIII.	

	\bibitem{Co}      H.D. Cornean. {\it On the Lipschitz continuity of spectral bands of Harper-like and magnetic Schr\"odinger operators}. Ann.\ H.\ Poincar{\'e} {\bf 11}(5), 973--990 (2010).
	
	\bibitem{CHP1}    H.D. Cornean, B. Helffer and R. Purice. \textit{Low lying spectral gaps induced by slowly varying magnetic fields}. J.\ Funct.\ Anal.\ {\bf 273}(1), 206--282 (2017).
	
	\bibitem{CHP2}    H.D. Cornean, B. Helffer and R. Purice. \textit{Peierls' substitution for low lying spectral energy windows}. To appear in Journal of Spectral Theory (2018).
	
	\bibitem{CHP2018} H.D. Cornean, B. Helffer and R. Purice. \textit{A Beals criterion for magnetic pseudodifferential operators proved with magnetic Gabor frames}. To appear in Commun.\ Part.\ Diff.\ Eq.\ (2018).
		
	\bibitem{CP-1}    H.D. Cornean and R. Purice. {\it On the Regularity of the Hausdorff
		Distance Between Spectra of Perturbed Magnetic Hamiltonians}. {Spectral
		Analysis of Quantum Hamiltonians}, in Operator Theory: Advances and Applications
	Volume 224, 2012, 55--66, Springer Basel.
	
	\bibitem{CP-2}    H.D. Cornean and R. Purice. {\it Spectral edge regularity of magnetic Hamiltonians}.  J.\ London Math.\ Soc.\ {\bf 92}(1), 89--104 (2015). 
	
	\bibitem{dNL}     G. De Nittis and M. Lein. {\it Applications of Magnetic $\Psi$DO
		Techniques to SAPT}.
	Rev.\ Math.\ Phys.\ {\bf 23}(3), 233-–260, (2011).
	
	\bibitem{FrTe}    S. Freund and S. Teufel. {\it Peierls substitution for magnetic Bloch bands}. Analysis and PDE {\bf 9}(4), 773--811 (2016).
		
	\bibitem{HS1}     B. Helffer and J. Sj\"{o}strand. {\it Analyse semi-classique pour
		l'{\'e}quation
		de Harper (avec application {\`a} l'{\'e}quation de Schr\"odinger avec champ magn{\'e}tique)}.
	M{\'e}moire de la SMF, {\bf No 34; Tome 116}, Fasc.4 (1988).
	
	\bibitem{HS}      B. Helffer and J. Sj\"ostrand. {\it Equation de Schr\"odinger avec
		champ magn{\'e}tique et {\'e}quation de
		Harper}. In LNP {\bf 345}, Springer-Verlag, Berlin, Heidelberg and New York,
	118--197 (1989).
	
	\bibitem{Hof}     D.R. Hofstadter. {\it Energy levels and wave functions of Bloch electrons in rational and irrational magnetic fields}. Phys. Rev. B. {\bf 14}(6), 2239--2249 (1976).
	
	\bibitem{Ho1}     L. H\"ormander. {\it The Analysis of Linear Partial
		Differential Operators, I}. (2-nd edition)\;Springer-Verlag, New York, 1990.
	
	\bibitem{Ho3}     L. H\"ormander. {\it The Analysis of Linear Partial
		Differential Operators, III}. (2-nd edition)\;Springer-Verlag, New York, 1994.
	
	\bibitem{IMP07}   V. Iftimie, M. M\u{a}ntoiu and  R. Purice. \textit{Magnetic pseudodifferential
operators}.  Publications of RIMS, {\bf 43}(3), 585-–623 (2007). 

	\bibitem{IMP10}   V. Iftimie, M. M{\u a}ntoiu and R. Purice. {\it Commutator criteria for magnetic pseudodifferential operators}.  Comm.\ Part.\ Diff.\ Eq.\ {\bf 35}(6),  1058--1094 (2010). 

	\bibitem{IP11}    V. Iftimie and R. Purice.  \textit{Magnetic Fourier integral operators}.  J.\ Pseudo.-Differ.\ Oper.\ Appl.\, {\bf 2}(2), 141--218 (2011).

	\bibitem{IP14}    V. Iftimie and R. Purice. {\it The Peierls-Onsager effective
		Hamiltonian in a complete gauge covariant setting: determining the
		spectrum}. J. Spectral Theory {\bf 5}(3), 445-531 (2015) 
		
	\bibitem{KO} M.V. Karasev and T.A. Osborn. \textit{Sympletic Areas, Quantization and Dynamics in Electromagnetic Fields}. J.\ Math.\ Phys.\ {\bf 43}, 756-788 (2002).
		
	\bibitem{Ka}      T. Kato. \textit{Pertubation Theory for Linear Operators}. (reprint of 2-nd edition)\;Springer-Verlag Berlin Heidelberg, 1995. 
		
	\bibitem{Lu}      J.M. Luttinger. {\it The Effect of a Magnetic Field on Electrons
		in a Periodic Potential}. Phys. Rev. {\bf 84}(4), 814--817 (1951).
	
	\bibitem{ML}      M. Lein, M. M\u antoiu and S. Richard. {\it Magnetic pseudodifferential operators with coefficients in $C^*$-algebras}. Publ.\ Res.\ Inst.\ Math.\ Sci.\ {\bf 46}(4), 755-788 (2010).
	
	\bibitem{MP} M. M\u antoiu and R. Purice \textit{The Magnetic Weyl Calculus}. J.\ Math.\ Phys.\ {\bf 45}, 1394-1417 (2004).
	
	\bibitem{MPR1}    M. M\u antoiu, R. Purice and S. Richard. {\it Spectral and
		Propagation Results for Magnetic Schr\"odinger Operators; a
		$C^*$-Algebraic Approach}. J.\ Funct.\ Anal.\ {\bf 250}(1), 42--67 (2007).	
	
	\bibitem{Ne-RMP}  G. Nenciu. {\it Dynamics of Bloch electrons in electric and
		magnetic fields: rigorous justification of the effective Hamiltonians}. Rev.\ Mod.\ Phys.\
	{\bf 63}(1), 91--127 (1991).
	
	\bibitem{Ne-JMP}  G. Nenciu. {\it On asymptotic perturbation theory for quantum mechanics: Almost invariant subspaces and gauge invariant magnetic perturbation theory}. J.\ Math.\ Phys.\ {\bf 43}(3), 1273--1298 (2002).
		
	\bibitem{Ne05}    G. Nenciu. {\it On the smoothness of gap boundaries for
		generalized Harper operators}. In Advances in operator algebras and
		mathematical physics. Theta Ser.\ Adv.\ Math., 5, 173-–182, Theta, Bucharest, 2005.
	
	\bibitem{PST}     G. Panati, H. Spohn and S. Teufel. {\it Effective dynamics for Bloch
		electrons: Peierls substitution and beyond}. Comm.\ Math.\ Phys.\ {\bf 242}(3), 547--578 (2003).
	
	\bibitem{Pe}     R.E. Peierls. {\it Quantum Theory of Solids}, Oxford University
	Press, 1955.
	
	\bibitem{S}    J. Sj\"ostrand. \textit{Microlocal analysis for the periodic magnetic Schr\" odinger equation and related questions}. CIME Lectures July 1989, in \textit{Microlocal Analysis and Applications}, LNM vol. 1495, 237--332, Springer, 1991. 
	
	
	
\end{thebibliography}
\end{document}